\documentclass[a4paper,12pt]{article}
\usepackage[utf8]{inputenc}
\usepackage{amsmath}
\usepackage{amssymb}
\usepackage{amsthm}
\usepackage[arrow, matrix, curve]{xy}
\usepackage{color}
\usepackage{enumerate}
\usepackage{a4wide}
\usepackage{hyperref}

\makeatletter

\DeclareMathOperator{\GL}{GL}
\DeclareMathOperator{\gl}{\mathfrak{gl}}
\DeclareMathOperator{\PGL}{PGL}

\DeclareMathOperator{\PSL}{PSL}
\let\sl\relax
\DeclareMathOperator{\sl}{\mathfrak{sl}}
\DeclareMathOperator{\upO}{O}
\DeclareMathOperator{\SO}{SO}
\DeclareMathOperator{\so}{\mathfrak{so}}

\let\sp\relax
\DeclareMathOperator{\sp}{\mathfrak{sp}}

\DeclareMathOperator{\su}{\mathfrak{su}}

\DeclareMathOperator{\upU}{U}
\DeclareMathOperator{\Str}{Str}
\DeclareMathOperator{\str}{\mathfrak{str}}
\DeclareMathOperator{\Aut}{Aut}
\DeclareMathOperator{\aut}{\mathfrak{aut}}
\DeclareMathOperator{\Co}{Co}
\DeclareMathOperator{\co}{\mathfrak{co}}

\newcommand{\fraka}{\mathfrak{a}}
\newcommand{\frake}{\mathfrak{e}}
\newcommand{\frakf}{\mathfrak{f}}
\newcommand{\frakg}{\mathfrak{g}}

\newcommand{\frakk}{\mathfrak{k}}
\newcommand{\frakl}{\mathfrak{l}}
\newcommand{\frakm}{\mathfrak{m}}

\newcommand{\frakp}{\mathfrak{p}}
\newcommand{\frakq}{\mathfrak{q}}
\newcommand{\fraks}{\mathfrak{s}}
\newcommand{\frakt}{\mathfrak{t}}
\newcommand{\fraku}{\mathfrak{u}}

\newcommand{\CC}{\mathbb{C}}

\newcommand{\FF}{\mathbb{F}}
\newcommand{\HH}{\mathbb{H}} 

\newcommand{\OO}{\mathbb{O}}

\newcommand{\RR}{\mathbb{R}}
\newcommand{\ZZ}{\mathbb{Z}}

\newcommand{\calB}{\mathcal{B}}

\newcommand{\calF}{\mathcal{F}}
\newcommand{\calG}{\mathcal{G}}
\newcommand{\calH}{\mathcal{H}}

\newcommand{\calO}{\mathcal{O}}

\newcommand{\calS}{\mathcal{S}}

\newcommand{\0}{\textbf{0}}

\DeclareMathOperator{\Ind}{Ind}

\DeclareMathOperator{\tr}{tr}
\DeclareMathOperator{\rank}{rank}
\DeclareMathOperator{\ad}{ad}
\DeclareMathOperator{\Ad}{Ad}

\DeclareMathOperator{\End}{End}

\DeclareMathOperator{\Sym}{Sym}
\DeclareMathOperator{\Herm}{Herm}
\DeclareMathOperator{\Skew}{Skew}

\DeclareMathOperator{\id}{id}
\DeclareMathOperator{\sgn}{sgn}
\DeclareMathOperator{\triv}{triv}

\renewcommand\Re{\operatorname{Re}}
\renewcommand\Im{\operatorname{Im}}

\renewcommand{\min}{{\textup{min}}}
\DeclareMathOperator{\otimeshat}{\widehat{\otimes}}

\theoremstyle{plain}
\newtheorem{theorem}{Theorem}[section]
\newtheorem{proposition}[theorem]{Proposition}
\newtheorem{lemma}[theorem]{Lemma}
\newtheorem{corollary}[theorem]{Corollary}

\newtheorem{fact}[theorem]{Fact}

\newtheorem{thmalph}{Theorem}

\theoremstyle{definition}

\newtheorem{remark}[theorem]{Remark}

\numberwithin{equation}{section}

\title{Exceptional theta correspondences via Plancherel formulas for rank one symmetric spaces}

\author{Jan Frahm and Quentin Labriet}
\date{\small\today}

\begin{document}
	
	\maketitle
	
	\begin{abstract}
		We consider the minimal representation of (a finite cover of) the conformal group of a simple split Jordan algebra over $\RR$ or $\CC$, whenever it exists. The conformal group contains a natural dual pair $G\times G'$, where $G$ is essentially the automorphism group of the Jordan algebra and $G'$ is either $\PSL(2,\RR)$, $\PGL(2,\RR)$ or $\PGL(2,\CC)$. The groups $G$ that arise in this way include the complex exceptional group of type $F_4$ as well as its compact and split real form.\\
        We explicitly determine the direct integral decomposition of the minimal representation restricted to the corresponding cover of $G\times G'$. This yields a one-to-one correspondence between certain representations of $G$ and (a finite cover of) $G'$. The representations of $G$ that occur in this correspondence are in the support of the Plancherel measure for a rank one symmetric space for $G$, and the proof makes use of the corresponding Plancherel formula.
	\end{abstract}



%\cleardoublepage

\setcounter{tocdepth}{1}

%\tableofcontents

%\cleardoublepage

\section*{Introduction}

The classical theta correspondence is a one-to-one correspondence between certain representations of two groups $G$ and $G'$ that form a dual pair inside a symplectic group. It is obtained by restricting the metaplectic representation of the metaplectic group, a double cover of the symplectic group, to the product of $G$ and $G'$ (or rather their corresponding coverings). The groups $G$ and $G'$ that can occur in this way are all classical groups.

In order to also obtain correspondences for exceptional groups, and also to enlarge the general framework of the theta correspondence, one can replace the metaplectic representation by a minimal representation of another simple Lie group and consider its restriction to dual pairs. This approach has received quite some attention during the past years, and many \emph{exceptional theta correspondences} have been studied. Most results in the literature are for dual pairs where one member is compact (see e.g. \cite{GS98,HPS96,Li96,Lok99,Lok00,Lok03,RS95}), and there are only few results, mostly partial, where both members are non-compact (see e.g. \cite{GLPS25,GS03,Li97,Li99,LS07,LS19}). Moreover, in contrast to the classical theta correspondence, which is a general statement for all dual pairs inside symplectic groups, most exceptional theta correspondence are studied case-by-case, i.e. for a specific dual pair inside a particular simple Lie group.

In this paper, we obtain novel exceptional theta correspondences for a family of dual pairs inside the conformal groups of simple Jordan algebras. The first member $G$ of the dual pairs is essentially the automorphism group of the Jordan algebra, which is often non-compact. The second member $G'$ is either $\PSL(2,\RR)$, $\PGL(2,\RR)$ or $\PGL(2,\CC)$. By using the $L^2$-model for the minimal representation of the conformal group, we relate the theta correspondence to the Plancherel formula for a rank one symmetric space for the automorphism group. Our correspondences can be seen as an archimedean analogue of some results by Savin~\cite{Sav94} for $p$-adic groups.

\paragraph{Statement of the results.} Let $V$ be a split simple Jordan algebra over $\FF=\RR$ or $\CC$ of rank at least two. There is a finite covering of the conformal group $\calG$ of $V$ which has a minimal representation $\Pi_\min$ (unless $\calG=\SO(p,q)$ with $p+q$ odd). The groups obtained in this way are listed in Section~\ref{sec:Tables}. The list includes the complex exceptional group of type $E_7$ as well as its Hermitian and split real forms $E_{7(-25)}$ and $E_{7(7)}$. Some groups $\calG$ have more than one minimal representation; we focus on one of them (see Section~\ref{sec:MinRep} for this choice).

Inside $\calG$ there is a natural dual pair $G\times G'$, where $G$ is essentially the automorphism group of $V$. If $V$ is real and Euclidean, then $G$ is compact and $G'\simeq\PSL(2,\RR)$, and if $V$ is real and non-Euclidean or complex, then $G$ is non-compact and $G'\simeq\PGL(2,\RR)$ or $\PGL(2,\CC)$. For instance, for $\calG$ of type $E_7$, the group $G$ is of type $F_4$, either compact, split or complex. In all cases, the restriction of the minimal representation $\Pi_\min$ to the corresponding finite cover of $G\times G'$ factors through $G\times\widetilde{G'}$ for some finite cover $\widetilde{G'}$ of $G'$.

Our first main result relates the restriction of the minimal representation $\Pi_\min$ to $G\times\widetilde{G'}$ to the Plancherel formula for a certain symmetric space for $G$. More precisely, we let $\calO_G$ denote the $G$-orbit of a primitive idempotent in $V$. It turns out to be a symmetric space for $G$ of split rank one, compact if $V$ is Euclidean and non-compact otherwise (see Lemma~\ref{lem:OHisSymmetricSpace} and Lemma~\ref{lem:MaxTorus}).

\begin{thmalph}[see Theorem~\ref{thm:AbstractTheta} and Section~\ref{sec:ExplicitTheta}]\label{introthmA}
    Let
    $$ L^2(\calO_G) \simeq \int_{\widehat{G}}\pi\,d\mu(\pi) $$
    denote the direct integral decomposition of the left regular representation of $G$ on $L^2(\calO_G)$ into irreducible unitary representations of $G$, then the restriction of $\Pi_\min$ to $G\times\widetilde{G'}$ decomposes as
    $$ \Pi_\min|_{G\times\widetilde{G'}} \simeq \int_{\widehat{G}}\pi\boxtimes\theta(\pi)\,d\mu(\pi) $$
    with $\theta(\pi)$ an irreducible unitary representation of $\widetilde{G'}$, and the map $\pi\mapsto\theta(\pi)$ is one-to-one (possibly ignoring a set of measure zero).
\end{thmalph}

By making the Plancherel formula for $\calO_G$ explicit, we obtain an explicit exceptional theta correspondence. The statement involves the two structure constants $r$ and $d$ of $V$ which are introduced in Section~\ref{sec:ConfGrpAlg} and listed in Section~\ref{sec:Tables} for all cases. The notation for representations of $G$ is explained in Section~\ref{sec:ExplicitTheta} and for representations of $G'$ and $\widetilde{G'}$ in Section~\ref{sec:RepPGL2}.

\begin{thmalph}[see Corollaries~\ref{cor:ExplicitThetaEuclidean}, \ref{cor:ExplicitThetaNonEuclidean} and \ref{cor:ExplicitThetaComplex}]\label{introthmB}
    \begin{enumerate}[(1)]
        \item If $V$ is real and Euclidean, then 
        $$ L^2(\calO_G) \simeq \widehat{\bigoplus_{k\in2\ZZ_{\geq0}}} \pi_{k\beta}, $$
        where $\pi_{k\beta}$ is the irreducible finite-dimensional representation of $G$ of highest weight $k\beta$ ($\beta$ and/or $2\beta$ are certain restricted roots, see Proposition~\ref{prop:RootSystem}). Moreover,
        $$ \theta(\pi_{k\beta})=\tau_{k+\frac{rd}{2}}^{\textup{hds}} $$
        is the holomorphic discrete series representation of a finite cover of $G'=\PSL(2,\RR)$ of parameter $k+\frac{rd}{2}>1$.
        \item If $V$ is real and non-Euclidean, then
        $$ L^2(\calO_G) \simeq \bigoplus_{\xi\in\ZZ/2\ZZ}\int^\oplus_{i\RR_+}\pi_{\xi,\lambda}\,d\lambda \oplus \widehat{\bigoplus_{k+\frac{rd}{2}\in2\ZZ_{>0}}} A_{\frakq}(k\beta), $$
        where $\pi_{\xi,\lambda}$ are certain (degenerate) principal series representations of $G$ and $A_{\frakq}(k\beta)$ Zuckerman's derived functor modules for some parabolic subalgebra $\frakq$ of $\frakg_\CC$. Moreover,
        $$ \theta(\pi_{\xi,\lambda})=\tau_{\xi,\lambda/2} \qquad \mbox{and} \qquad \theta(A_{\frakq}(k\beta))=\tau^{\textup{ds}}_{k+\frac{rd}{2}}, $$
        where $\tau_{\xi,\nu}$ is a unitary principal series representation of $G'=\PGL(2,\RR)$ and $\tau_m^{\textup{ds}}$ is the discrete series representation of parameter $m>1$ (except for the case $\calG=\SO(p+1,q+1)$ with $p-q\equiv2\pmod4$ where $\theta(\pi_{\xi,\lambda})=\tau_{\xi+1,\lambda/2}$).
        \item If $V$ is complex, then
        $$ L^2(\calO_G) \simeq \widehat{\bigoplus_{m\in\ZZ}} \int^\oplus_{i\RR_+} \pi_{m,\lambda}\,d\lambda, $$
        where $\pi_{m,\lambda}$ are certain (degenerate) principal series representations of $G$. Moreover,
        $$ \theta(\pi_{m,\lambda})=\tau_{m,\lambda/2} $$
        is a unitary principal series representation of $G'=\PGL(2,\CC)$.
    \end{enumerate}
\end{thmalph}

\paragraph{Method of proof.}

We use the $L^2$-model of $\Pi_\min$ that was constructed by Vergne--Rossi~\cite{VR76} for $V$ Euclidean and by Dvorsky--Sahi~\cite{DS99} and Kobayashi--{\O}rsted~\cite{KO03c} for the other cases. In this model, the representation space is $L^2(\calO)$, where $\calO\subseteq V$ is the subvariety of rank one elements (with an additional positivity condition if $V$ is Euclidean). The restriction of $\Pi_\min$ to $G$ is the left regular representation of $G$ on $L^2(\calO)$. We show that $\calO$ contains an open dense subset of the form $\RR_+\times\calO_G$, $\RR^\times\times\calO_G$ or $\CC^\times\times\calO_G$ for $V$ real Euclidean, real non-Euclidean and complex, respectively (see Proposition~\ref{prop:FOHopendense}). This relates the decomposition of $\Pi_\min|_{G\times\widetilde{G'}}$ to the Plancherel formula for $\calO_G$. To also keep track of the action of $G'$, we use the Lie algebra representation which was studied in \cite{HKM14}. The key computation is Lemma~\ref{lem:KeyLemma} where the underlying Lie algebra representation of $G'$ on each $G$-isotypic component is determined. This essentially proves Theorem~\ref{introthmA}.

To show Theorem~\ref{introthmB}, we make the Plancherel formula for $\calO_G$ explicit using results from the literature (see e.g. \cite{Mol92,vdB05}). In the case where $V$ is real and non-Euclidean, a few additional arguments using $K$-types and intertwining operators are necessary to obtain the theta correspondence (see Remark~\ref{rem:AbstractTheta} and Section~\ref{sec:PlancherelNonEuclidean}).

We remark that some parts of this paper could be significantly shortened using case-by-case computations instead of the framework of Jordan algebras. However, in this way, we also treat the exceptional groups along the same lines.

\paragraph{Relation to other work.}

For $V$ Euclidean, our result can mostly be reduced to the classical theta correspondence and is well-known (see e.g. \cite{KV78}). We decided to include this case, because the proof is structurally identical to the other cases, the main difference being that $\calO_G$ is compact. All other correspondences seem to be new, except for the discrete spectrum in the case $\calG=\SO(p+1,q+1)$ which is a special case of \cite[Theorem C]{KO03}.

Similar results for the same algebraic groups, but over non-archimedean local fields have been obtained by Savin~\cite{Sav94}. Instead of decomposing the unitary representation $\Pi_\min$, he considers quotients of the smooth vectors $\Pi_\min^\infty$, just as in the classical theta correspondence. It would be interesting to investigate these questions also in the archimedean case.

Our proof builds on the \emph{stratified model} by the second author, which was applied in the context of holomorphic discrete series in our previous work~\cite{FL25}.

\section{Minimal representations of conformal groups}

We recall the construction of $L^2$-models for minimal representations of conformal groups of simple split Jordan algebras over $\RR$ and $\CC$. For this, we follow the exposition in \cite{HKM14}.

\subsection{The conformal group and its Lie algebra}\label{sec:ConfGrpAlg}

Let $V$ denote a simple unital Jordan algebra over $\FF\in\{\RR,\CC\}$. In the case $\FF=\CC$, we often view $V$ as a real Jordan algebra and note that as such it is also simple. If $\FF=\RR$ then its complexification $V_\CC$ is a simple complex Jordan algebra, and if $\FF=\CC$ then $V_\CC$ is isomorphic to $V\oplus V$. In some of the formulas below we use $\delta=\dim_\RR\FF\in\{1,2\}$ to make the formulation uniform.

We write $n=\dim_\FF V$ for the dimension of $V$ and $r=\mathrm{rk}_\FF V$ for its rank, and we assume $r\geq2$. For $x\in V$ let $L(x):V\to V$ denote multiplication by $x$. The bilinear form
$$ \tau_\FF:V\times V\to\FF, \quad (x,y)\mapsto\frac{r}{n}\tr(L(x)L(y)) $$
is non-degenerate and normalized such that $\tau_\FF(e,e)=r$, where $e\in V$ is the unit element. For an endomorphism $T\in\End(V)$ we write $T^\#$ for its adjoint with respect to $\tau_\FF$. The real part of $\tau_\FF$,
$$ \tau(x,y) = \Re\tau_\FF(x,y), $$
is a non-degenerate bilinear form on the real vector space $V$ which is positive definite if and only if $V$ is a Euclidean Jordan algebra.

Let
$$ P:V\to\End(V), \quad P(x)=2L(x)^2-L(x^2) $$
denote the quadratic representation of $V$ and $P(x,y)=L(x)L(y)+L(y)L(x)-L(xy)$ its bilinear version. The structure group $\Str(V)$ of $V$ is defined to be the group of all $g\in\GL(V)$ such that $P(gx)=gP(x)g^\#$ for all $x\in V$. Its Lie algebra $\str(V)$ is the direct sum of $L(V)=\{L(x):x\in V\}$ and $\aut(V)$, the Lie algebra of derivations of $V$, i.e. all $T\in\gl(V)$ such that $T(x\cdot y)=Tx\cdot y+x\cdot Ty$ for all $x,y\in V$. The latter one is spanned by elements of the form $[L(x),L(y)]$. Moreover, $\aut(V)$ is the Lie algebra of the group $\Aut(V)$ of automorphisms of $V$, i.e. those $g\in\GL(V)$ such that $g(x\cdot y)=gx\cdot gy$ for all $x,y\in V$. We note that by \cite[Proposition II.4.2]{FK94} and \cite[(1.8) and (1.9)]{HKM14} we have
\begin{equation}
    |\det(g)| = 1 \qquad \mbox{for all }g\in\Aut(V).\label{eq:DetTrivialOnAut}
\end{equation}

Note that $L(x)^\#=L(x)$ for $x\in V$ and that $D^\#=-D$ for all $D\in\aut(V)$. We have the following characterization of $\Aut(V)$ and $\aut(V)$:

\begin{lemma}[{see \cite[Proposition VIII.2.4 and VIII.2.6]{FK94}}]\label{lem:CharacterizationH}
    For $g\in\Str(V)$ the following equivalences hold:
    $$ g\in\Aut(V) \qquad \Leftrightarrow \qquad ge=e. $$
    For $X\in\str(V)$ the following equivalences hold:
    $$ X\in\aut(V) \qquad \Leftrightarrow \qquad Xe=0 \qquad \Leftrightarrow \qquad X^\#=-X. $$
\end{lemma}

The conformal group $\Co(V)$ is the group of rational transformations of $V$ generated by all translations $\overline{n}_a:V\to V,\,x\mapsto x+a$ ($x\in V$), the structure group $\Str(V)$ and the inversion $j:V\to V,\,x\mapsto-x^{-1}$. This is a simple real Lie group with trivial center. Its Lie algebra $\co(V)$ can be decomposed as
\begin{equation}
    \co(V)=V\oplus\str(V)\oplus V\label{eq:GelfandNaimark}
\end{equation}
with Lie bracket
\begin{equation}
    [(u_1,T_1,v_1),(u_2,T_2,v_2)] = (T_1u_2-T_2u_1,[T_1,T_2]+2(u_1\Box v_2)-2(u_2\Box v_1),-T_1^\# v_2+T_2^\# v_1),\label{eq:LieBracketInCO}
\end{equation}
where $u\Box v=L(uv)+[L(u),L(v)]\in\str(V)$.

In order to discuss some structure theory of the group $\Co(V)$ we choose a Cartan involution of a particular form. We first choose a Cartan involution $\vartheta$ of $V$, i.e. $\vartheta\in\Aut(V)$ is an automorphism of order two such that the $\RR$-bilinear form
$$ V\times V\to\RR, \quad (x,y)\mapsto(x|y):=\tau(x,\vartheta y) $$
is positive definite. The adjoint $g^*$ of $g\in\Str(V)$ with respect to this inner product is given by $g^*=\vartheta\circ g^\#\circ\vartheta$, so $g\mapsto\theta(g)=(g^*)^{-1}$ defines a Cartan involution on $\Str(V)$. It can be extended to a Cartan involution on $\Co(V)$ by $\theta(g)=\vartheta\circ j\circ g\circ j\circ\vartheta$. The corresponding involution on the Lie algebra $\co(V)$ is given by
$$ \theta(u,T,v) = (-\vartheta(v),-T^*,-\vartheta(u)) \qquad (u,v\in V,T\in\str(V)). $$
In particular the corresponding Cartan decomposition $\frakg=\frakk\oplus\frakp$ is given by
\begin{align}
    \frakk &= \{(u,T,-\vartheta(u)):u\in V,T+T^*=0\},\label{eq:LieAlgK}\\
    \frakp &= \{(u,T,\vartheta(u)):u\in V,T=T^*\}.\label{eq:LieAlgP}
\end{align} 

We write
$$ V=V^+\oplus V^- $$
for the decomposition of $V$ into $+1$ and $-1$ eigenspaces of $\vartheta$. Since $\vartheta$ is an automorphism, we have $V^\pm\cdot V^\pm\subseteq V^+$ and $V^\pm\cdot V^\mp\subseteq V^-$. In particular, $V^+$ is a Euclidean Jordan algebra. Let $c_1,\ldots,c_r\in V$ be a Jordan frame in $V^+$, i.e. a complete set of pairwise orthogonal primitive idempotents. If $V(c_i,\lambda)\subseteq V$ denotes the eigenspace of $L(c_i)$ to the eigenvalue $\lambda\in\{0,\frac{1}{2},1\}$, we have the following Peirce decomposition of $V$:
$$ V = \bigoplus_{1\leq i\leq j\leq r} V_{ij}, $$
where
$$ V_{ii} = V(c_i,1) \qquad \mbox{and} \qquad V_{ij} = V(c_i,\tfrac{1}{2})\cap V(c_j,\tfrac{1}{2}) \qquad \mbox{for }i\neq j. $$
We assume in what follows that the Jordan algebra $V$ is \emph{split} over $\FF$, i.e. $V_{ii}=\FF c_i$ for all $i=1,\ldots,r$. All other Peirce spaces $V_{ij}$ ($i\neq j$) have a common dimension $d=\dim_\FF V_{ij}$.

The Peirce decomposition has the following properties:
\begin{align}
    V_{ij}\cdot V_{jk} &\subseteq V_{ik} && \mbox{for $i,j,k$ distinct,}\label{eq:VijkMultiplication}\\
    V_{ij}\cdot V_{ij} &\subseteq V_{ii}\oplus V_{jj} && \mbox{for $i,j$ distinct.}\label{eq:VijMultiplication}
\end{align}
With the notation $V_{ij}^\pm=V_{ij}\cap V^\pm$ we even have the following result about multiplication operators:

\begin{lemma}\label{lem:VijkIso}
    Let $i,j,k$ be distinct and $x\in V_{jk}^\pm$ non-zero. Then $L(x):V_{ij}\to V_{ik}$ is an isomorphism.
\end{lemma}

\begin{proof}
    Let $x\in V_{jk}$ and $y\in V_{ij}$, then by \cite[Lemma IV.2.2]{FK94} (whose proof literally translates to split Jordan algebras over $\FF=\RR,\CC$ by replacing the norm $\|x\|^2$ by $\tau_\FF(x,x)$) we have
    \begin{equation}
        L(x)^2y = \frac{1}{8}\tau_\FF(x,x)y.\label{eq:Lxsquare}
    \end{equation}
    If additionally $x\in V_{jk}^\pm$, then $\vartheta(x)=\pm x$ and hence $\tau_\FF(x,x)=\tau(x,x)=\pm\tau(x,\vartheta(x))=\pm\|x\|^2$. It follows that $L(x):V_{ij}\to V_{ik}$ is an isomorphism with inverse $\pm8\|x\|^{-2}L(x):V_{ik}\to V_{ij}$.
%    Pairing both sides with $z\in V_{ij}$ using the trace form yields
%    \begin{equation}
%        \tau(xy,xz) = \tau(x(xy),z)=\frac{1}{8}\tau(x,x)\tau(y,z) \qquad (x\in V_{jk}^\pm,y,z\in V_{ij}).\label{eq:IsometryLxOnVij}
%    \end{equation}
%    Now, for $x\in V_{jk}^\pm$ we have $\vartheta(x)=\pm x$ and hence $\tau(x,x)=\pm\tau(x,\vartheta(x))=\pm\|x\|^2$. This shows
%    $$ \tau(xy,xz) = \pm\frac{1}{8}\|x\|^2\tau(y,z) \qquad (x\in V_{jk}^\pm,y,z\in V_{ij}). $$
%    Putting $z=\vartheta(y)$ and using once more that $x=\pm\vartheta(x)$ yields
%    $$ \|xy\|^2 = \frac{1}{8}\|x\|^2\|y\|^2 \qquad (x\in V_{jk}^\pm,y\in V_{ij}). $$
%    Hence, $L(x):V_{ij}\to V_{ik}$ is a scalar multiple of an isometry and in particular an isomorphism.
\end{proof}

By \cite[Proposition 1.4.6]{Moe10}, we have the following root space decomposition of $\frakl=\str(V)$ with respect to the abelian subalgebra $\bigoplus_{j=1}^r\RR L(c_j)$:
$$ \frakl = \frakl_0 \oplus \bigoplus_{i\neq j}\frakl_{ij}, $$
where
$$ \frakl_0 = \{L(x)+D:x\in\bigoplus_{i=1}^{r}V_{ii},D\in\aut(V),Dc_i=0\,\forall\,i\} $$
and
$$ \frakl_{ij} = \{\tfrac{1}{2}L(x)+[L(c_i),L(x)]:x\in V_{ij}\}. $$
For $x\in V_{ij}$ we have $[L(c_i),L(x)]=-[L(c_j),L(x)]$ by \cite[Lemma 1.4.4~(2)]{Moe10}, so the Lie algebra $\frakg=\aut(V)$ decomposes as
\begin{equation}
    \frakg = \frakg_0 \oplus \bigoplus_{i<j}\frakg_{ij}\label{eq:DecompositionLieAlgH}
\end{equation}
with
$$ \frakg_0 = \frakl_0\cap\frakg = \{D\in\aut(V):Dc_i=0\,\forall\,i\} \qquad \mbox{and} \qquad \frakg_{ij} = (\frakl_{ij}\oplus\frakl_{ji})\cap\frakg = [L(c_i),L(V_{ij})]. $$
Note the following inclusions for $i,j,k,\ell$ distinct:
\begin{gather}\label{eq:CommutationInclusions}
    [\frakg_{ij},\frakg_{ij}]\subseteq\frakg_0, \qquad [\frakg_{ij},\frakg_{jk}]\subseteq\frakg_{ik}, \qquad [\frakg_{ij},\frakg_{k\ell}]=\{0\},\\
    [\frakg_0,\frakg_{ij}]\subseteq\frakg_{ij}, \qquad [\frakg_0,\frakg_0]\subseteq\frakg_0.\nonumber
\end{gather}

\subsection{The minimal representation}\label{sec:MinRep}

We recall the $L^2$-model for minimal representations constructed in \cite{HKM14,Moe10}. Note that we slightly modify the formulas by twisting the representation with the Cartan involution in \cite[equation (2.4)]{HKM14}. The construction can be summarized as follows.

Let $c$ be a primitive idempotent and $\calO=\Str(V)_0\cdot c$ its orbit of the identity component $\Str(V)_0$ of $\Str(V)$. Then $\calO=-\calO$ if and only if $V$ is not Euclidean. $\calO$ carries a unique (up to scalar multiples) equivariant measure $\mu$ satisfying $d\mu(gx)=|\det_\RR(g)|^{\frac{rd}{2n}}\,d\mu(x)$. In particular it behaves under dilation by $z\in\FF^\times$ as $d\mu(zx)=|z|^{\delta\cdot\frac{rd}{2}}\,d\mu(x)$.

Unless $\co(V)\simeq\so(p+1,q+1)$ with $p,q\geq2$ and $p+q$ odd, there is a unique irreducible unitary representation $\Pi_\min$ of a finite covering of $\Co(V)_0$ on $L^2(\calO):=L^2(\calO,d\mu)$ whose restriction to $\Str(V)_0\overline{N}$ splits and is given by
\begin{align}
    \Pi_\min(\overline{n}_a)f(x) &= e^{-i\tau(x,b)}f(x) && (\overline{n}_b=\exp(0,0,b)\in \overline{N}),\label{eq:ActionN}\\
    \Pi_\min(l)f(x) &=|\mathrm{det}_\RR(l)|^{-\frac{rd}{4n}}f(l^{-1}x) && (l\in\Str(V)_0).\label{eq:ActionL}
\end{align}
The corresponding Lie algebra representation of $\mathfrak{g}$ is given by regular differential operators and hence extends to $C^\infty(\calO)$. It is given by
\begin{align*}
    d\Pi_\min(0,0,u) &= -i\tau(x,u) && (u\in V),\\
    d\Pi_\min(0,T,0) &= -\partial_{Tx}-\frac{rd}{4n}\tr_\RR(T) && (T\in\mathfrak{l}),\\
    d\Pi_\min(v,0,0) &= -i\tau(\calB,v) && (v\in V),
\end{align*}
where $\calB:C^\infty(\calO)\to C^\infty(\calO)\otimes V_\CC$ is the Bessel operator given by
$$ \calB = \sum_{\alpha,\beta}P(\widehat{e}_\alpha,\widehat{e}_\beta)x\frac{\partial^2}{\partial x_\alpha\partial x_\beta}+\frac{\delta d}{2}\sum_\alpha\widehat{e}_\alpha\frac{\partial}{\partial x_\alpha} $$
with $(e_\alpha)_\alpha\subseteq V$ a basis of $V$ (as real vector space), $(\widehat{e}_\alpha)_\alpha\subseteq V$ its dual basis with respect to $\tau$ and $x=\sum_\alpha x_\alpha e_\alpha$.

\begin{remark}
    Note that for $\FF=\CC$ this is not exactly the version of $\Pi_\min$ as constructed in \cite{HKM14} since the bilinear form $\tau$ is normalized differently. The relation between our formulas and the ones in \cite{HKM14} is the automorphism $(u,T,v)\mapsto(2u,T,\frac{1}{2}v)$.
\end{remark}

If $V$ is real and Euclidean or complex, then we consider $\Pi_\min$ as an irreducible unitary representation of a finite cover $\widetilde{\calG}$ of the connected group $\calG=\Co(V)_0$. If $V$ is real and non-Euclidean, we extend $\Pi_\min$ to an irreducible unitary representation of a finite cover $\widetilde{\calG}$ of the (possibly disconnected) group $\calG=\Co(V)_0\cup(-\id_V)\Co(V)_0$ in order to be sure that $\calG$ has the disconnected group $\PGL(2,\RR)$ as subgroup, instead of only $\PSL(2,\RR)$ (see Section~\ref{sec:DualPair}).

\begin{lemma}
    If $V$ is real and non-Euclidean, the representation $\Pi_\min$ of a finite cover of $\Co(V)_0$ can be extended to a finite cover of $\calG=\Co(V)_0\cup(-\id_V)\Co(V)_0$ by
    $$ \Pi_\min(-\id_V)f(x) = f(-x) \qquad (f\in L^2(\calO),x\in\calO). $$
\end{lemma}

\begin{proof}
    It is easy to check that if we define $\Pi_\min(-\id_V)$ in the above way, the following identity holds for all $u,v\in V$ and $T\in\frakl$:
    $$ \Pi_\min(-\id_V)\circ d\Pi_\min(u,T,v) = d\Pi_\min(-u,T,-v)\circ\Pi_\min(-\id_V). $$
    Since $\Ad(-\id_V)(u,T,v)=(-u,T,-v)$, this shows the claim.
\end{proof}

\section{Dual pairs and orbits}

We identify a dual pair of the form $G\times G'$ in $\calG$, where $G=\calG\cap\Aut(V)$ and $G'$ has Lie algebra $\sl(2,\FF)$, and study orbits of $G$ in $\calO$.

\subsection{The dual pair}\label{sec:DualPair}

We first identify the dual pair on the Lie algebra level. The Lie algebra of $G=\calG\cap\Aut(V)$ is $\frakg=\aut(V)$. We write $\frakg'$ for the Lie algebra spanned over $\FF$ by
$$ E=(e,0,0), \qquad F=(0,0,e), \qquad \mbox{and} \qquad H=(0,2L(e),0). $$
By \eqref{eq:LieBracketInCO} these elements satisfy the standard commutator relations
$$ [H,E] = 2E, \qquad [H,F] = -2F, \qquad [E,F] = H, $$
so $\frakg'\simeq\sl(2,\FF)$.

\begin{lemma}\label{lem:DualPairLieAlg}
    $(\frakg,\frakg')$ is a dual pair in $\co(V)$.
\end{lemma}

\begin{proof}
    Let first $X=(u,T,v)$ be in the centralizer of $\frakg'\simeq\sl(2,\FF)$. By \eqref{eq:LieBracketInCO} we find $(u,0,-v)=[H,X]=0$, so $u=v=0$, and further $(Te,0,0)=[X,E]=0$, so $Te=0$ and hence $T\in\aut(V)=\frakg$ by Lemma~\ref{lem:CharacterizationH}.\\
    Now let $X=(u,T,v)$ be in the centralizer of $\aut(V)$, then by \eqref{eq:LieBracketInCO}
    $$ (Su,[S,T],-S^\#v) = [(0,S,0),X] = 0 \qquad \mbox{for all }S\in\aut(V). $$
    In particular, $Su=Sv=0$ for all $S\in\aut(V)=\aut(V)^\#$. In the case where $V$ is complex, \cite[Lemma VIII.5.1]{FK94} implies that $u,v\in\CC e$. If $V$ is real, then $V_\CC$ is simple and $\aut(V_\CC)=\aut(V)_\CC$ by \cite[Proposition VIII.1.1]{FK94}, so $u,v\in\CC e\subseteq V_\CC$ by the same argument. But $V\cap\CC e=\RR e$, so $u,v\in\RR e$ in this case. It remains to show that $T\in\FF L(e)$. Write $T=L(x)+D$ with $x\in V$ and $D\in\aut(V)$. Then
    $[S,T]=L(Sx)+[S,D]=0$ for all $S\in\aut(V)$, which implies that $Sx=0$ and $[S,D]=0$ for all $S$. Since $\aut(V)$ is semisimple, this implies $D=0$, and by using once more \cite[Lemma VIII.5.1]{FK94} we find $x\in\FF e$.
\end{proof}

Now let $G'$ denote the subgroup of $\calG$ given by the fractional linear transformations of the form
$$ z\mapsto(az+b)(cz+d)^{-1} \qquad \begin{pmatrix}a&b\\c&d\end{pmatrix}\in\begin{cases}\PSL(2,\RR)&\mbox{if $V$ is Euclidean,}\\ 
\PGL(2,\FF)&\mbox{if $V$ is non-Euclidean or complex.}\end{cases} $$

Since $-\id_V\in \calG$ for $V$ real and non-Euclidean, this disconnected group is indeed a subgroup of $\calG$. Clearly $G'\simeq\PSL(2,\RR)$ resp. $\PGL(2,\RR)$ resp. $\PGL(2,\CC)$ for $V$ Euclidean resp. non-Euclidean resp. complex. The Lie algebra of $G'$ is $\frakg'$.

To prove that $(G,G')$ is a dual pair in $\calG$, we first consider the pair $(\Aut(V),\PGL(2,\FF))$ inside $\Co(V)$, where $\PGL(2,\FF)$ is defined as above, also for $V$ Euclidean. For $r>2$, the automorphism group $\Aut(V)$ has trivial center and $(\Aut(V),\PGL(2,\FF))$ turns out to be a dual pair inside $\Co(V)$. For $r=2$, the center $Z(\Aut(V))$ of $\Aut(V)$ has two elements, the non-trivial one being the linear map that maps the identity element $e$ of $V$ to itself and acts by $-1$ on the orthogonal complement $e^\perp\subseteq V$. In this case, the centralizer of $\Aut(V)$ is the product of $\PGL(2,\FF)$ and $Z(\Aut(V))$.

\begin{lemma}
    Let $V$ be a simple split Jordan algebra over $\FF\in\{\RR,\CC\}$ of rank $r$.
    \begin{itemize}
        \item If $r>2$, then $(\Aut(V),\PGL(2,\FF))$ is a dual pair in $\Co(V)$ and $(G,G')$ is a dual pair in $\calG$.
        \item If $r=2$, then $(\Aut(V),\PGL(2,\FF) Z(\Aut(V)))$ is a dual pair in $\Co(V)$ and $(G,G'Z(G))$ is a dual pair in $\calG$.
    \end{itemize}
\end{lemma}

\begin{proof}
    Since elements of $\Aut(V)$ are multiplicative and fix the identity, it follows immediately that $\Aut(V)$ commutes with all fractional linear transformations in $\PGL(2,\FF)$. It remains to show that $\Aut(V)$ and $\PGL(2,\FF)\times Z(\Aut(V))$ are mutual centralizers.
    
    Let $g\in Z_{\Co(V)}(\PGL(2,\FF))\subseteq Z_{\Co(V)}(\frakg')$, then $\Ad(g)H=H$. Since $H$ defines the three-grading \eqref{eq:GelfandNaimark}, it follows that $g$ is contained in the Levi subgroup of the corresponding maximal parabolic subgroup, i.e. $g\in\Str(V)$. Moreover, $\Ad(g)E=E$ implies that $ge=e$ and hence $g\in\Aut(V)$ by Lemma~\ref{lem:CharacterizationH}.

    Now assume that $g\in Z_{\Co(V)}(\Aut(V))$. Since $\Aut(V)$ commutes with $\PGL(2,\FF)$, we must have $\Ad(g)\frakg'\subseteq Z_\frakg(\aut(V))=\frakg'$ by Lemma~\ref{lem:DualPairLieAlg}. So $\Ad(g)\frakg'=\frakg'$ and hence $\Ad(g)|_{\frakg'}\in \Aut(\frakg') \simeq \PGL_2 (\FF)$. We can therefore choose $g'\in\PGL(2,\FF)\subseteq\Co(V)$ such that $\Ad(g)=\Ad(g')$ on $\frakg'$. It follows that $g^{-1}g'\in Z_{\Co(V)}(\frakg')=\Aut(V)$. But since $g^{-1}g'$ also commutes with every element in $\Aut(V)$, it has to be contained in the center of $\Aut(V)$. If $r>2$, the center of $\Aut(V)$ is trivial, so $g=g'\in\PGL(2,\FF)$. For $r=2$, the center of $\Aut(V)$ also contains the non-trivial element which acts by $-1$ on the orthogonal complement $e^\perp\subseteq V$ of the identity $e$.

    The same arguments apply to $(G,G')$ in $\calG$, because they have the same Lie algebras as $(\Aut(V),\PGL(2,\FF))$.
\end{proof}

\subsection{Orbits of $G$ in $\calO$}

We now study orbits of $G$ in $\calO$. Note that the action of $G$ commutes with dilations by $\FF^\times$ and $\calO$ is $\FF^\times$-stable (resp. $\RR_+$ if $V$ is Euclidean). Moreover, the map $\calO \to \FF,\ z\mapsto\tau_\FF(z,e)$ is constant on each $G$-orbit and has value $1$ at $c$. This implies that $\calO_G=G\cdot c$ and its dilations by $\FF^\times$ (resp. $\RR_+$) are distinct $G$-orbits in $\calO$. Their union turns out to be open and dense.

\begin{proposition}\label{prop:FOHopendense}
\begin{enumerate}[(1)]
    \item If $\FF=\RR$ and $V$ is Euclidean, then $\RR_+\calO_G$ is open dense in $\calO$.
    \item If $\FF=\RR$ and $V$ is non-Euclidean, then $\RR^\times\calO_G$ is open dense in $\calO$.
    \item If $\FF=\CC$, then $\CC^\times\calO_G$ is open dense in $\calO$.
\end{enumerate}
\end{proposition}

\begin{proof}
    Write $L=\Str(V)\cap\calG$ and $K_L=L^\theta$ for its maximal compact subgroup. By \cite[Proposition 1.5.5]{Moe10} the subgroup $Q_1=\{g\in L:gV(c,1)\subseteq V(c,1)\}$ is parabolic in $L$ with Langlands decomposition $Q_1=M_1N_1$, where $M_1=\{g\in L:gL(c)=L(c)g\}$ and $N_1$ fixes $c$. Moreover, the stabilizer $L_c$ of $c$ in $L$ is given by $L_c=(M_1\cap L_c)N_1\subseteq Q_1$ and $M_1=(M_1\cap K_L)A_1(M_1\cap L_c)$ with $A_1=\exp(\RR L(c)+\RR L(e-c))$. It follows that
    $$ Q_1\cdot c = (M_1\cap K_L)A_1(M_1\cap L_c)N_1\cdot c = (M_1\cap K_L)A_1\cdot c = (M_1\cap K_L)\cdot\RR_+ c. $$
    Since $M_1\cap K_L$ is compact and leaves $V(c,1)=\FF c$ invariant, its orbit through $c$ is $\{c\}$ for $V$ real and Euclidean, $\{\pm c\}$ for $V$ real and non-Euclidean, and $\{e^{i\theta}c:\theta\in\RR\}$ for $V$ complex. This implies $Q_1\cdot c=\FF^\times c$, so $L/Q_1=\calO/\FF^\times$. To show the claim of the proposition it therefore suffices to show that the orbit of $G$ through the base point of $L/Q_1$ is open and dense.\\
    This is equivalent to showing that the $Q_1$-orbit through the base point of the semisimple symmetric space $L/G$ is open and dense. The subgroup $G\subseteq L$ is open in the fixed point set of the involution $g\mapsto g^{-\#}$ which commutes with the Cartan involution $g\mapsto\theta(g)=g^{-*}=\vartheta\circ g^{-\#}\circ\vartheta$. The intersection of both $(-1)$-eigenspaces in $\frakl$ is $L(V^+)$ which contains $\sum_{i=1}^r\RR L(c_i)$ as a maximal abelian subspace, where $c_1,\ldots,c_r$ is a Jordan frame in $V^+$ (see \cite[Lemma 1.4.5]{Moe10}). Note that this subspace is also maximal abelian in the structure algebra $\str(V^+)$ of the Euclidean Jordan subalgebra $V^+$. The corresponding root systems are both of type $A_{r-1}$ (see \cite[Proposition 1.4.6]{Moe10}), so they have the same Weyl group. Together with \cite[Proposition 8.1]{vdB05} this shows that $Q_1$ has a unique open (hence dense) orbit in $L/G$.
\end{proof}

In the next section we show that $\calO_G$ is a semisimple symmetric space, so it carries a unique (up to scalar multiples) $G$-invariant measure $d\mu_G$. We relate the restriction of the measure $d\mu$ on $\calO$ to the open subset $\FF^\times\calO_G\simeq\FF^\times\times\calO_G$ to the measure $d\mu_G$.

\begin{lemma}\label{lem:MeasureInPolarCoordinates}
    After appropriate normalization of the measures $d\mu$ and $d\mu_G$, the following integral formula holds for all $f\in C_c(\FF^\times\calO_G)$:
    $$ \int_\calO f(x)\,d\mu(x) = \int_{\FF^\times}\int_{\calO_G}f(ty)\,d\mu_G(y)\,|t|^{\delta\frac{rd}{2}}\,d^\times t, $$
    where $d^\times t$ is the multiplication invariant measure on $\FF^\times$. In particular,
    $$ L^2(\calO) \simeq L^2(\calO_G)\otimeshat L^2(\FF^\times,|t|^{\delta\frac{rd}{2}}\,d^\times t). $$
\end{lemma}

\begin{proof}
    The character by which the measure $d\mu$ transforms is a power of the absolute value of the determinant, hence it is trivial on $G$ by \eqref{eq:DetTrivialOnAut} and the measure $d\mu$ is $G$-invariant. It follows that $d\mu=\chi(t)\,d^\times t\,d\mu_G$ for some function $\chi$ on $\FF^\times$. Moreover, $d\mu(zx)=|z|^{\delta\frac{rd}{2}}\,d\mu(x)$ and the claim follows.
\end{proof}

\subsection{Structure of \texorpdfstring{$\calO_G$}{OH}}

Let $G_c$ denote the stabilizer of $c$ in $G$, so that $\calO_G\simeq G/G_c$. The pair $(G,G_c)$ turns out to be a symmetric pair, so $\calO_G$ is a semisimple symmetric space.

\begin{lemma}\label{lem:OHisSymmetricSpace}
    Let $w=\exp(\pi(c,0,-c))$, then $w\in G$ is the identity on $V(c,1)\oplus V(c,0)$ and minus the identity on $V(c,\frac{1}{2})$. Moreover, $\sigma:G\to G,\,h\mapsto whw^{-1}$ defines an involution on $G$ and the fixed point subgroup $G^\sigma$ of $G$ agrees with $G_c$, except in the case $r=2$ where $G_c$ is a subgroup of $G^\sigma$ of index two. In particular, $\calO_G\simeq G/G_c$ is a semisimple symmetric space, compact if $V$ is Euclidean and non-compact if $V$ is non-Euclidean.
\end{lemma}

\begin{proof}
    It follows from the formula \eqref{eq:LieBracketInCO} for the Lie bracket of $\co(V)$ that
    $$\exp(t\ad(c,0,-c))(0,\id,0)=(-\tfrac{1}{2}\sin(2t)c,\id+(\cos(2t)-1)L(c),-\tfrac{1}{2}\sin(2t)c),$$
    so that $\Ad(w)(0,\id,0)=(0,\id,0)$. But since $(0,\id,0)\in\co(V)$ induces the grading \eqref{eq:GelfandNaimark}, this implies that $w\in\calG\cap\Str(V)$. Moreover, a similar computation shows that $w$ acts on $V(c,1)\oplus V(c,0)$ by $+1$ and on $V(c,\frac{1}{2})$ by $-1$, so it is an automorphism.
    
    It is clear that $\sigma$ defines an involution on $G$. We now show that $G_c=G^\sigma$ for $r>2$ and that $G_c\subseteq G^\sigma$ is of index two for $r=2$. Assume first that $h\in G^\sigma$, then $h$ leaves $V(c,1)\oplus V(c,0)$ invariant. If $r>2$, then $V(c,1)$ and $V(c,0)$ are both simple and non-isomorphic, so the automorphism $h$ must map $V(c,1)$ onto itself. Finally, $hc$ must be an idempotent in $V(c,1)=\FF c$, but $c$ is the only idempotent in $\FF c$. It follows that $hc=c$, so $h\in G_c$. For $r=2$, one can easily construct an element $h_0$ in $G^\sigma$ that maps $c$ to $e-c$ and vice versa. The same arguments as above show that either $h$ is contained in $G_c$ or in $hG_c$. If conversely $h\in G_c$, then $hL(c)h^{-1}=L(hc)=L(c)$. This implies that $h$ leaves the eigenspaces of $L(c)$ invariant, hence it has to commute with $w$, so $h\in G^\sigma$.
\end{proof}

Abusing notation, we also write $\sigma$ for the involution on the Lie algebra $\frakg$ of $G$. To study the symmetric space $\calO_G=G/G_c$ in more detail, we write $\frakg=\frakg^\sigma\oplus\frakg^{-\sigma}$ for the decomposition into the eigenspaces of $\sigma$. The Cartan involution $\theta$ on $\str(V)$ restricts to a Cartan involution on $\frakg$ that commutes with $\sigma$, because $\vartheta(c)=c$. Note that $\theta$ is trivial on $\frakg$ if $V$ is Euclidean since in this case $G$ is compact.

\begin{lemma}\label{lem:MaxTorus}
    We have
    $$ \frakg^{-\sigma} = [L(c),L(V(c,\tfrac{1}{2}))] \qquad \mbox{and} \qquad \frakg^{-\sigma}\cap\frakg^{\pm\theta} = [L(c),L(V(c,\tfrac{1}{2})^\pm)]. $$
    Moreover, $\calO_G$ is a symmetric space of split rank $1$ and rank $\delta=\dim_\RR\FF\in\{1,2\}$.
\end{lemma}

\begin{proof}
    We complete $c_1=c$ to a Jordan fram $c_1,\ldots,c_r$ of $V$ and recall the corresponding decomposition \eqref{eq:DecompositionLieAlgH} of $\frakg$. Since $w$ acts by $+1$ on $V_{ij}$ if either $i=j=1$ or $i,j>1$ and by $-1$ if $i=1$ and $j\neq1$, the involution $\sigma$ clearly acts by $-1$ on $\frakg_{1j}=[L(c_1),L(V_{1j})]$ for $j>1$ and by $+1$ on $\frakg_{ij}=[L(c_i),L(V_{ij})]$ for $i,j>1$. Moreover, each $D\in\frakg_0$ preserves the subspaces $V_{ij}$ and is therefore fixed by the involution $\sigma$. This shows that
    $$ \frakg^{-\sigma}=\bigoplus_{j=2}^{r}\frakg_{1j} = [L(c_1),L(V(c_1,\tfrac{1}{2}))]. $$
    The Cartan involution $\vartheta$ fixes $c_1$, so the corresponding Cartan involution $\theta$ of $\frakl$ maps $[L(c_1),L(x)]$ to $[L(c_1),L(\vartheta x)]$. This shows the formula for $\frakg^{-\sigma}\cap\frakg^{\pm\theta}$.

    To show the remaining claims, we let $x\in V_{12}^\pm$, $x\neq0$, and show that $\RR[L(c_1),L(x)]$ is maximal abelian in $\frakg^{-\sigma}\cap\frakg^{\pm\theta}$. So let $y\in V(c_1,\frac{1}{2})$ with
    $$ \big[[L(c_1),L(x)],[L(c_1),L(y)]\big]=0. $$
    %For this, we first note that
    %$$ [[L(c),L(x)],[L(c),L(y)]] = -\tfrac{1}{4}[L(x),L(y)]+\tfrac{1}{2}[L(c),L((xy)_1)]. $$
    Write $y=\sum_{i=2}^r y_{1i}$ with $y_{1i}\in V_{1i}$. We first consider $i\neq2$, then $[L(c_1),L(x)]c_i=0$ and $[L(c_1),L(y)]c_i=\frac{1}{4}y_{1i}$, so
    $$ 0 = \big[[L(c_1),L(x)],[L(c_1),L(y)]\big]c_i = \frac{1}{4}[L(c_1),L(x)]y_{1i} = \frac{1}{4}(c_1(xy_{1i})-\tfrac{1}{2}xy_{1i}) = -\frac{1}{8}xy_{1i} $$
    since $xy_{1i}\in V_{2i}$ by \eqref{eq:VijkMultiplication}. By Lemma~\ref{lem:VijkIso}, $L(x):V_{1i}\to V_{2i}$ is an isomorphism, so $y_{1i}=0$ for all $i>2$ and hence $y=y_{12}\in V_{12}$. Next, we apply $\big[[L(c_1),L(x)],[L(c_1),L(y)]\big]$ to $x$. By \eqref{eq:VijMultiplication}, $xy\in V_{11}\oplus V_{22}=\FF c_1\oplus\FF c_2$ we can write $xy=\lambda c_1+\mu c_2$. Moreover, $\lambda=\tau(xy,c_1)=\tau(x,yc_1)=\frac{1}{2}\tau(x,y)=\tau(x,yc_2)=\tau(xy,c_2)=\mu$, so $xy=\frac{1}{2}\tau(x,y)(c_1+c_2)$. Using this, we find 
    \begin{align*}
        0 = \big[[L(c_1),L(x)],[L(c_1),L(y)]\big]x &= \frac{1}{4}\tau(x,y)[L(c_1),L(x)](c_1-c_2) - \frac{1}{4}\tau(x,x)[L(c_1),L(y)](c_1-c_2)\\
        &= -\frac{1}{8}\tau(x,y)x+\frac{1}{8}\tau(x,x)y.
    \end{align*}
    Since $x\in V_{12}^\pm$, it follows that $\tau(x,x)=\pm(x|x)\neq0$, so $x$ and $y$ are linearly dependent.
\end{proof}

If $V$ is Euclidean, then $\calO_G$ is a compact symmetric space, otherwise it is a non-compact pseudo-Riemannian symmetric space which we now study in more detail in terms of a maximal compact and a maximal split torus in $\frakg^{-\sigma}$. For this, recall the decomposition \eqref{eq:DecompositionLieAlgH} of $\frakg$ into $\frakg_0$ and $\frakg_{ij}$ ($i<j$) as well as the commutator inclusions \eqref{eq:CommutationInclusions}. Let $x\in V_{12}^\pm\setminus\{0\}$ and put $D_0=[L(c_1),L(x)]\in\frakg_{12}$. We first collect some formulas for commutators between $D_0\in \frakg_{12}$ and elements in the above decomposition.

\begin{lemma}\label{lem:D0commutators}
    \begin{enumerate}[(1)]
        \item For $y\in V_{12}$ with $\tau_\FF(x,y)=0$ we have
        $$ [D_0,[L(c_1),L(y)]] = -\frac{1}{4}[L(x),L(y)] \quad \mbox{and} \quad [D_0,[L(x),L(y)]] = \frac{1}{2}\tau(x,x)[L(c_1),L(y)]. $$
        \item For $y\in V_{1j}$ and $z\in V_{2j}$, $j>2$, we have
        $$ [D_0,[L(c_1),L(y)]] = -\frac{1}{2}[L(c_2),L(xy)] \quad \mbox{and} \quad [D_0,[L(c_2),L(z)]] = \frac{1}{2}[L(c_1),L(xz)]. $$
        \item $[D_0,\frakg_{ij}]=\{0\}$ for $i,j>2$.
    \end{enumerate}
\end{lemma}

\begin{proof}
    (3) is clear by the commutator inclusions \eqref{eq:CommutationInclusions}. Let us show (1). Since $D_0c_1=-\frac{1}{4}x$ and $D_0y\in V_{11}\oplus V_{22}$ by \eqref{eq:VijMultiplication}, we have
    $$ [D_0,[L(c_1),L(y)]] = [L(D_0c_1),L(y)] + [L(c_1),L(D_0y)] = -\frac{1}{4}[L(x),L(y)]+0, $$
    because $[L(c_1),L(V_{11})]=[L(c_1),L(V_{22})]=\{0\}$ (use \cite[Proposition II.1.1]{FK94} with $y=c_1$). Moreover, by the commutator inclusions \eqref{eq:CommutationInclusions} we have
    $$ [D_0,[L(x),L(y)]] = [L(c_1),L(z)] $$
    for some $z\in V_{12}$. To determine $z$ we apply both sides to $c_1$. The right hand side gives $-\frac{1}{4}z$ while the left hand sides equals
    $$ D_0[L(x),L(y)]c_1-[L(x),L(y)]D_0c_1 = 0+\frac{1}{4}[L(x),L(y)]x = \frac{1}{4}(L(x)^2-L(x^2))y. $$
    Since $x^2=\frac{1}{2}\tau(x,x)(c_1+c_2)$ and $xy=\frac{1}{2}\tau_\FF(x,y)(c_1+c_2)=0$ by \eqref{eq:VijMultiplication}, it follows that $z=\frac{1}{2}\tau(x,x)y$.

    For (2) we argue similarly as above. We know that $[D_0,\frakg_{1j}]\subseteq\frakg_{2j}$ and $[D_0,\frakg_{2j}]\subseteq\frakg_{1j}$ and then apply $[L(c_1),L(y)]$ and $[L(c_1),L(z)]$ to $c_j$ to show the claimed formulas.
\end{proof}

\begin{proposition}\label{prop:RootSystem}
    \begin{enumerate}[(1)]
        \item\label{prop:RootSystem1} Let $h_0\in V_{12}^-$ with $\tau(h_0,h_0)=-32$ and put $H_0=[L(c_1),L(h_0)]$ and $\fraka=\RR H_0$. Then the root system of $(\frakg,\fraka)$ is of the form $\{\pm\alpha\}$, $\{\pm2\alpha\}$ or $\{\pm\alpha,\pm2\alpha\}$, where $\alpha(H_0)=1$, with root spaces
        \begin{align*}
            \frakg^0 &= \{D\in\frakg_0:Dh_0=0\}\oplus\FF H_0\oplus\bigoplus_{3\leq i<j\leq r}\frakg_{ij},\\
            \frakg^{\pm\alpha} &= \left\{[L(c_1),L(y)]\mp\frac{1}{2}[L(c_2),L(h_0y)]:y\in\bigoplus_{j=3}^{r}V_{1j}\right\},\\
            \frakg^{\pm2\alpha} &= \{[L(c_1),L(y)]\mp\frac{1}{8}[L(h_0),L(y)]:y\in V_{12},\tau_\FF(h_0,y)=0\}.
        \end{align*}
        In particular, $\dim_\FF\frakg^{\pm\alpha}=(r-2)d$ and $\dim_\FF\frakg^{\pm2\alpha}=d-1$, so
        $$ \rho_\fraka = \frac{1}{2}\left((\dim_\RR\frakg^{\alpha})\alpha+(\dim_\RR\frakg^{2\alpha})2\alpha\right) = \delta\left(\frac{rd}{2}-1\right)\alpha. $$
        \item\label{prop:RootSystem2} Assume $\FF=\RR$. Let $t_0\in V_{12}^+$ with $\tau(t_0,t_0)=32$ and put $T_0=[L(c_1),L(t_0)]$ and $\frakt=\RR T_0$. Then the root system of $(\frakg_\CC,\frakt_\CC)$ is of the form $\{\pm\beta\}$, $\{\pm2\beta\}$ or $\{\pm\beta,\pm2\beta\}$, where $\beta(iT_0)=1$, with root spaces
        \begin{align*}
            \frakg_\CC^0 &= \{D\in\frakg_{0,\CC}:Dt_0=0\}\oplus\CC T_0\oplus\bigoplus_{3\leq i<j\leq r}(\frakg_{ij})_\CC,\\
            \frakg_\CC^{\pm\beta} &= \left\{[L(c_1),L(y)]\mp\frac{1}{2}i[L(c_2),L(t_0y)]:y\in\bigoplus_{j=3}^{r}(V_{1j})_\CC\right\},\\
            \frakg_\CC^{\pm2\beta} &= \{[L(c_1),L(y)]\mp\frac{1}{8}i[L(t_0),L(y)]:y\in (V_{12})_\CC,\tau(t_0,y)=0\}.
        \end{align*}
        %In particular, $\exp(tT_0)=1$ if and only if
        %$$ t\in\begin{cases}2\pi\ZZ&\mbox{if $r>2$,}\\\pi\ZZ&\mbox{if $r=2$.}\end{cases} $$
        In particular, $\dim\frakg_{\CC}^{\pm\beta}=(r-2)d$ and $\dim\frakg_{\CC}^{\pm2\beta}=d-1$, so 
        $$ \rho_\frakt = \frac{1}{2}\left((\dim\frakg_{\CC}^{\beta})\beta+(\dim\frakg_{\CC}^{2\beta})2\beta\right) = \left(\frac{rd}{2}-1\right)\beta. $$
    \end{enumerate}
\end{proposition}

\begin{proof}
    Use Lemma~\ref{lem:D0commutators} and \eqref{eq:Lxsquare}.
\end{proof}

Finally, in the setting of Proposition~\ref{prop:RootSystem}~\eqref{prop:RootSystem2}, we also determine the intersection of the torus $\exp(\frakt)$ with the stabilizer $G_c$.

\begin{lemma}\label{lem:TorusStabilizer}
    Assume $\FF=\RR$, let $t_0\in V_{12}^+$ with $\tau(t_0,t_0)=32$ and put $T_0=[L(c_1),L(t_0)]$. Then $\exp(tT_0)\in G_c$ if and only if $t\in\pi\ZZ$.
\end{lemma}

\begin{proof}
    Using $t_0^2=\frac{1}{2}\tau(t_0,t_0)(c_1+c_2)=16(c_1+c_2)$, a short computation reveals that $T_0(c_1-c_2)=-\frac{1}{2}t_0$ and $T_0t_0=8(c_1-c_2)$, so
    $$ \exp(sT_0)(c_1-c_2) = \cos(2s)(c_1-c_2)-\frac{1}{4}\sin(2s)t_0. $$
    Since $\exp(sT_0)(c_1+c_2)=c_1+c_2$ for all $s\in\RR$, this shows the claim.
\end{proof}

\section{Theta correspondence vs. Plancherel formula}

In this section we relate the decomposition of $\Pi_\min$ into irreducible representations of $G\times G'$ to the Plancherel formula for $L^2(\calO_G)$. For this we first recall the representation theory of $\PGL(2,\FF)$ and then study intertwining operators between representations of $G\times G'$ and $\Pi_\min|_{G\times G'}$.

\subsection{Representations of $\PSL(2,\FF)$ and $\PGL(2,\FF)$}\label{sec:RepPGL2}

We recall the representation theory of $\PGL(2,\FF)$ in order to fix notation (see e.g. \cite[Chapter 2, \S4]{Kna86} for details). For $m\in\ZZ$ and $\nu\in\CC$ we consider the principal series representation $\sigma_{m,\nu}$ of $\PGL(2,\FF)$ induced from the character
$$ \begin{pmatrix}a&b\\0&1\end{pmatrix}\mapsto |a|_m^\nu := |a|^\nu\left(\frac{a}{|a|}\right)^m. $$
Note that for $\FF=\RR$ the representation only depends on $m$ modulo $2$, so we will consider $m$ as an element of $\ZZ/2\ZZ$ in this case.

\begin{fact}
    The infinite-dimensional irreducible unitary representations of $\PGL(2,\FF)$ are:
    \begin{itemize}
        \item The unitary principal series $\sigma_{m,\nu}$ with $\nu\in i\RR$ and $m\in\ZZ/2\ZZ$ for $\FF=\RR$ and $m\in\ZZ$ for $\FF=\CC$. The only equivalences among them are $\sigma_{m,\nu}\simeq\sigma_{-m,-\nu}$.
        \item The complementary series $\sigma_{m,\nu}$ with $\nu\in(-\frac{\delta}{2},\frac{\delta}{2})\setminus\{0\}$ and $m\in\ZZ/2\ZZ$ for $\FF=\RR$ and $m=0$ for $\FF=\CC$, where $\delta=\dim_\RR\FF\in\{1,2\}$. The only equivalences among them are $\sigma_{m,\nu}\simeq\sigma_{-m,-\nu}$.
        \item For $\FF=\RR$ the discrete series of parameter $k\geq2$ even, which is the quotient of $\sigma_{m,\nu}$ for $\nu=\frac{1-k}{2}$ and arbitrary $m\in\ZZ/2\ZZ$ by the kernel of the standard intertwining operator $A_{m,\nu}$.
    \end{itemize}
    All irreducible unitary representations are irreducible when restricted to a parabolic subgroup.
\end{fact}

To exhibit these representations in the theta correspondence, we use explicit realizations. By restriction to the opposite unipotent radical, the representation $\sigma_{m,\nu}$ can be realized on a space of smooth functions on $\FF$ with the action
$$ \sigma_{m,\nu}\begin{pmatrix}a&b\\c&d\end{pmatrix}^{-1}f(z) = \left|\frac{(a+bz)^2}{ad-bc}\right|^{-\nu-\frac{\delta}{2}}_{-m}f\left(\frac{c+dz}{a+bz}\right), $$
where $\delta=\dim_\RR\FF$, i.e. $\delta=1$ for $\FF=\RR$ and $\delta=2$ for $\FF=\CC$. In particular, we find
$$ d\sigma_{m,\nu}\begin{pmatrix}0&u\\0&0\end{pmatrix} = \partial_{uz^2}+(2\nu+\delta)\Re(uz)+2mi\Im(uz). $$
%For $\FF=\RR$, the lowest $K$-type is the character $|\det|^m$ and spanned by the function
%$$ x\mapsto(1+x^2)^{-\nu-\frac{1}{2}}. $$
In this picture, the Hilbert space for the unitary principal series is simply $L^2(\FF)$ with respect to the Lebesgue measure $dz$ on $\FF$. For the invariant inner product on the complementary series and the discrete series, we need the standard intertwining operator $A_{m,\nu}:\sigma_{m,\nu}\to\sigma_{-m,-\nu}$ which is in this picture given by
$$ A_{m,\nu}f(w) = \int_\FF |z|^{2\nu-\delta}_{2m}f(w+z)\,dz. $$
The invariant inner product on complementary series and discrete series is induced by
$$ (f_1,f_2)\mapsto\int_\FF (A_{m,\nu}f_1)(z)\overline{f_2(z)}\,dz. $$

We denote by $\calF:\calS'(\FF)\to\calS'(\FF)$ the Euclidean Fourier transform given by
$$ \calF u(\zeta) = \int_\FF e^{-i\Re(z\zeta)}f(z)\,dz, $$
where we write $\zeta=\xi+i\eta\neq0$. Then $\widehat{\sigma}_{m,\nu}(g)=\calF\circ\sigma_{m,\nu}(g)\circ\calF^{-1}$ defines an equivalent representation on a subspace of $\calS'(\FF)$. In case $\widehat{\sigma}_{m,\nu}$ has a finite-dimensional subrepresentation, this subrepresentation consists of distributions supported at $\{0\}$ (i.e. the Fourier transforms of polynomials). If we replace $\widehat{\sigma}_{m,\nu}$ by its quotient by this finite-dimensional subrepresentation, we can consider the representation space to be a subspace of $\calS'(\FF^\times)$. The standard intertwining operator $\widehat{A}_{m,\nu}=\calF\circ A_{m,\nu}\circ\calF^{-1}:\widehat{\sigma}_{m,\nu}\to\widehat{\sigma}_{-m,-\nu}$ is up to scalar multiples given by
$$ \widehat{A}_{m,\nu}f(\zeta) = |\zeta|_{-2m}^{-2\nu}f(\zeta), $$
so the Hilbert spaces for complementary series representations $\widehat{\sigma}_{m,\nu}$ ($\nu\in(-\frac{\delta}{2},\frac{\delta}{2})$) and discrete series representations $\widehat{\sigma}_{m,\nu}$ ($\nu=\frac{1-k}{2}$, $k\geq2$ even) are given by $L^2(\FF^\times,|\zeta|^{-2\nu}\,d\zeta)$.

It is easy to see the action of the parabolic subgroup
$$ P' = \left\{\begin{pmatrix}a&0\\c&d\end{pmatrix}:ad\neq0\right\}\subseteq\PGL(2,\FF) $$
in this picture:
\begin{equation*}
    \widehat{\sigma}_{m,\nu}\begin{pmatrix}1&0\\v&1\end{pmatrix}f(\zeta) = e^{-i\Re(v\zeta)}f(\zeta), \qquad\qquad
    \widehat{\sigma}_{m,\nu}\begin{pmatrix}a&0\\0&d\end{pmatrix}f(\zeta) = \left|\frac{a}{d}\right|_m^{\nu-\frac{\delta}{2}}f\left(\frac{d}{a}\zeta\right).
\end{equation*}
To compute the action of the unipotent radical we note that
\begin{align*}
    \calF\circ\frac{d}{dx} &= i\xi\circ\calF, & \calF\circ x &= i\frac{d}{d\xi}\circ\calF,\\
    \calF\circ\frac{d}{dy} &= -i\eta\circ\calF, & \calF\circ y &= -i\frac{d}{d\eta}\circ\calF.
\end{align*}
Then a short computations shows that
\begin{align*}
    i\,d\widehat{\sigma}_{m,\nu}\begin{pmatrix}0&1\\0&0\end{pmatrix} &= \xi(\partial_\xi^2-\partial_\eta^2)+2\eta\partial_\xi\partial_\eta-(2\nu-\delta)\partial_\xi+2mi\partial_\eta,\\
    i\,d\widehat{\sigma}_{m,\nu}\begin{pmatrix}0&i\\0&0\end{pmatrix} &= 2\xi\partial_\xi\partial_\eta-\eta(\partial_\xi^2-\partial_\eta^2)-(2\nu-\delta)\partial_\eta-2mi\partial_\xi,
\end{align*}
where we interpret $\eta=0$ and $\partial_\eta=0$ in case $\FF=\RR$.

%For $\FF=\RR$, the function spanning the lowest $K$-type becomes the Bessel function
%$$ \xi\mapsto|\xi|^\nu K_\nu(|\xi|). $$

With the structure constants $r$ and $d$ of the Jordan algebra $V$ in mind, we conjugate the representation $\widehat{\sigma}_{m,\nu}$ by the linear isomorphism
$$ \Phi_{m,\nu}:L^2(\FF^\times,|\zeta|^{-2\nu}\,d\zeta) \to L^2(\FF^\times,|\zeta|^{\delta\frac{rd}{2}-\delta}\,d\zeta), \quad \Phi_{m,\nu} f(\zeta) = |\zeta|^{-(\nu+\delta\frac{rd}{4}-\frac{\delta}{2})}_{-m}f(\zeta) $$
to obtain $\tau_{m,\nu}(g)=\Phi_{m,\nu}\circ\widehat{\sigma}_{m,\nu}(g)\circ\Phi_{m,\nu}^{-1}$. The parabolic subgroup $P'\subseteq\PGL(2,\RR)$ acts by
\begin{equation}
    \tau_{m,\nu}\begin{pmatrix}1&0\\v&1\end{pmatrix}f(\zeta) = e^{-i\Re(v\zeta)}f(\zeta), \qquad\qquad
    \tau_{m,\nu}\begin{pmatrix}a&0\\0&d\end{pmatrix}f(\zeta) = \left|\frac{a}{d}\right|^{-\delta\frac{rd}{4}}f\left(\frac{d}{a}\zeta\right).\label{eq:GroupActionTau}
\end{equation}
Moreover, using
\begin{align*}
    \Phi_{m,\nu}\circ\xi &= \xi\circ\Phi_{m,\nu}, & \Phi_{m,\nu}\circ\partial_\xi &= \left(\partial_\xi+\frac{(\nu+\delta\tfrac{rd}{4}-\tfrac{\delta}{2})\xi-im\eta}{\xi^2+\eta^2}\right)\circ\Phi_{m,\nu},\\
    \Phi_{m,\nu}\circ\eta &= \eta\circ\Phi_{m,\nu}, & \Phi_{m,\nu}\circ\partial_\eta &= \left(\partial_\eta+\frac{(\nu+\delta\tfrac{rd}{4}-\tfrac{\delta}{2})\eta+im\xi}{\xi^2+\eta^2}\right)\circ\Phi_{m,\nu}.
\end{align*}
we obtain for $\FF=\RR$: 
\begin{equation}
    i\,d\tau_{m,\nu}\begin{pmatrix}0&1\\0&0\end{pmatrix} = \xi\partial_\xi^2+\frac{rd}{2}\partial_\xi+\frac{\left(\left(\frac{rd-2}{4}\right)^2-\nu^2\right)}{\xi},\label{eq:SL2Bessel}
\end{equation}
and for $\FF=\CC$:
\begin{align}
    i\,d\tau_{m,\nu}\begin{pmatrix}0&1\\0&0\end{pmatrix} &= \xi(\partial_\xi^2-\partial_\eta^2)+2\eta\partial_\xi\partial_\eta+rd\partial_\xi+\frac{\left(\left(\frac{rd-2}{2}\right)^2-\nu^2-m^2\right)\xi+2im\nu\eta}{\xi^2+\eta^2},\label{eq:SL2Bessel1}\\
    i\,d\tau_{m,\nu}\begin{pmatrix}0&i\\0&0\end{pmatrix} &= 2\xi\partial_\xi\partial_\eta-\eta(\partial_\xi^2-\partial_\eta^2)+rd\partial_\eta+\frac{\left(\left(\frac{rd-2}{2}\right)^2-\nu^2-m^2\right)\eta-2im\nu\xi}{\xi^2+\eta^2}.\label{eq:SL2Bessel2}
\end{align}

\begin{theorem}\label{thm:UnitaryDualPGL}
    The infinite-dimensional irreducible unitary representations of $\PGL(2,\FF)$ are:
    \begin{itemize}
        \item The unitary principal series $\tau_{m,\nu}$ with $\nu\in i\RR$ and $m\in\ZZ/2\ZZ$ for $\FF=\RR$ and $m\in\ZZ$ for $\FF=\CC$. The only equivalences among them are $\tau_{m,\nu}\simeq\tau_{-m,-\nu}$.
        \item The complementary series $\tau_{m,\nu}$ with $\nu\in(-\frac{\delta}{2},\frac{\delta}{2})\setminus\{0\}$ and $m\in\ZZ/2\ZZ$ for $\FF=\RR$ and $m=0$ for $\FF=\CC$.
        \item For $\FF=\RR$ the discrete series $\tau_k^{\textup{ds}}$ of parameter $k\geq2$ even which equals $\tau_{m,\nu}$ with $\nu=\frac{1-k}{2}$ and arbitrary $m\in\ZZ/2\ZZ$.
    \end{itemize}
    All these representations are unitary on $L^2(\FF^\times,|\zeta|^{\delta(\frac{rd}{2}-1)}\,d\zeta)$ and irreducible when restricted to the parabolic subgroup $P'$.
\end{theorem}

\begin{remark}
    For $\FF=\RR$ the restriction of the discrete series $\tau_k^{\textup{ds}}$ ($k\geq2$ even) to $\PSL(2,\RR)$ decomposes into the direct sum of two irreducible unitary representations, the holomorphic discrete series $\tau_k^{\textup{hds}}$ on $L^2(\RR_+,|\xi|^{\frac{rd}{2}-1}\,d\xi)$ and the antiholomorphic discrete series $\tau_k^{\textup{ahds}}$ on $L^2(\RR_-,|\xi|^{\frac{rd}{2}-1}\,d\xi)$. Moreover, the universal cover $\widetilde{\PSL}(2,\RR)$ allows holomorphic and antiholomorphic discrete series $\tau_k^{\textup{hds}}$ and $\tau_k^{\textup{ahds}}$ for any real number $k>1$, realized on the same Hilbert spaces by the same formulas for the action of $P'$ and the Lie algebra $\sl(2,\RR)$ (see e.g. \cite{Kos00}).
\end{remark}

\subsection{Intertwining operators}

We define a map
$$ T:C^\infty(\calO_G)\otimes C^\infty(\FF^\times)\to C^\infty(\FF^\times\calO_G), \quad T(f_1\otimes f_2)(tg\cdot c)=f_1(g\cdot c)f_2(t), $$
where we use $\RR_+$ instead of $\RR^\times$ if $V$ is Euclidean. By Lemma~\ref{lem:MeasureInPolarCoordinates}, it gives rise to an isometric isomorphism
\begin{equation}\label{eq:IntertwinerIso}
    T:L^2(\calO_G)\otimeshat L^2(\FF^\times,|\zeta|^{\delta\frac{rd}{2}}\,d^\times\zeta)\to L^2(\calO),
\end{equation}
again with $\RR^\times$ replaced by $\RR_+$ for $V$ Euclidean.

The map $T$ clearly intertwines the left-regular representation $\ell$ of $G$ on $C^\infty(\calO_G)$ with the restriction $\Pi_\min|_G$ by \eqref{eq:ActionL} and \eqref{eq:DetTrivialOnAut}. It is also easy to show that it intertwines the action of the parabolic subgroup $P'\subseteq G'$:

\begin{lemma}\label{lem:IntertwinerQandH}
    The map $T$ intertwines the representations $\ell\boxtimes\tau_{m,\nu}|_{P'}$ and $\Pi_\min|_{G\times P'}$ of $G\times P'$ for any $m\in\ZZ$, $\nu\in\CC$.
\end{lemma}

\begin{proof}
    We first verify the intertwining property for $\exp(\FF F)$:
    \begin{align*}
        \Pi_\min(\exp(sF))\circ T(f_1\otimes f_2)(tg\cdot c) &= e^{-i\tau(tg\cdot c,se)}T(f_1\otimes f_2)(tg\cdot c) = f_1(g\cdot c)\cdot e^{-i\Re(st)}f_2(t)\\
        &= T(f_1\otimes\tau_{m,\nu}(\exp(sF))f_2)(tg\cdot c).
    \end{align*}
    Now we consider the action of $\exp(\FF H)$:
    \begin{align*}
        \Pi_\min(\exp(sH))\circ T(f_1\otimes f_2)(tg\cdot c) &= e^{-\delta\frac{rd}{2}s}T(f_1\otimes f_2)(e^{-2s}tg\cdot c) = e^{-\delta\frac{rd}{2}s}f_1(g\cdot c)f_2(e^{-2s}t)\\
        &= f_1(g\cdot c)\cdot\tau_{m,\nu}(\exp(sH))f_2(t)\\
        &= T(f_1\otimes\tau_{m,\nu}(\exp(sH))f_2)(tg\cdot c).
    \end{align*}
    In the case where $\FF=\RR$ and $V$ is non-Euclidean, we also need to check the intertwining property for an element of the non-identity component of $P'\subseteq G'\simeq\PGL(2,\RR)$. We do this for $-\id_V\in G'$:
    \begin{align*}
        \Pi_\min(-\id_V)\circ T(f_1\otimes f_2)(tg\cdot c) &= T(f_1\otimes f_2)(-tg\cdot c) = f_1(g\cdot c)f_2(-t)\\
        &= f_1(g\cdot c)\cdot \tau_{m,\nu}(-\id_V)f_2(t) = T(f_1\otimes \tau_{m,\nu}(-\id_V)f_2)(tg\cdot c).\qedhere
    \end{align*}
\end{proof}

It remains to study the action of $\exp(\FF E)$. This will be done via the Lie algebra action of $\FF E$. To formulate the statement we let $C\in U(\frakg)$ denote the Casimir element of $\frakg$ with respect to the Killing form $B$ normalized by
\begin{equation}
    B([L(c_i),L(v)],[L(c_i),L(w)]) = \tau(v,w) \qquad (v,w\in V_{ij}).\label{eq:NormalizationKillingForm}
\end{equation}
More precisely, if $(X_\alpha)$ is a basis of the (real) Lie algebra $\frakg$ and $(X_\alpha')$ its dual basis with respect to $B$, then
$$ C=\sum_\alpha X_\alpha X_\alpha'. $$

In the case $\FF=\CC$ we have $\frakg_\CC\simeq\frakg\times\frakg$ by the two complex-linear embeddings
$$ \frakg\hookrightarrow\frakg_\CC=\{X+IY:X,Y\in\frakg\}, \quad X\mapsto\frac{1}{2}(X\pm IiX), $$
so there are two linearly independent invariants of order two in $U(\frakg)$, $C$ and
$$ D = \sum_\alpha X_\alpha\cdot(iX_\alpha') = \sum_\alpha (iX_\alpha)\cdot X_\alpha', $$
where again $(X_\alpha)$ is a basis of the real Lie algebra $\frakg$ and $(X_\alpha')$ its dual basis with respect to $B$.

\begin{lemma}[Key Lemma]\label{lem:KeyLemma}
    \begin{enumerate}[(1)]
        \item For $\FF=\RR$:
        $$ i\,d\Pi_\min(e,0,0)\circ T(f_1\otimes f_2) = T(f_1\otimes(\xi f_2''+\tfrac{rd}{2}f_2')) + 8T(d\ell(C)f_1\otimes \xi^{-1}f_2). $$
        \item For $\FF=\CC$:
        \begin{align*}
            i\,d\Pi_\min(e,0,0)\circ T(f_1\otimes f_2) ={}& T\Big(f_1\otimes\big[\xi(\partial_\xi^2-\partial_\eta^2)f_2+2\eta\partial_\xi\partial_\eta f_2+rd\partial_\xi f_2\big]\Big)\\
            & \qquad+ 8T\Big(d\ell(C)f_1\otimes\tfrac{\xi}{\xi^2+\eta^2}f_2-d\ell(D)f_1\otimes\tfrac{\eta}{\xi^2+\eta^2}f_2\Big),\\
            i\,d\Pi_\min(ie,0,0)\circ T(f_1\otimes f_2) ={}& T\Big(f_1\otimes\big[2\xi\partial_\xi\partial_\eta f_2-\eta(\partial_\xi^2-\partial_\eta^2)f_2+rd\partial_\eta f_2\big]\Big)\\
            & \qquad+ 8T\Big(d\ell(C)f_1\otimes\tfrac{\eta}{\xi^2+\eta^2}f_2+d\ell(D)f_1\otimes\tfrac{\xi}{\xi^2+\eta^2}f_2\Big).
        \end{align*}
    \end{enumerate}
\end{lemma}

\begin{proof}
    We claim that it suffices to show the identities evaluated at $x=c$. Let us show this for $\FF=\RR$, the case $\FF=\CC$ is similar. For $s\in\RR$ and $g\in G$ we can write
    \begin{equation*}
        i\,d\Pi_\min(e,0,0)\circ T(f_1\otimes f_2)(e^{2s}g\cdot c) = ie^{\frac{rd}{2}s}\Pi_\min(\exp(-tH)g^{-1})\circ d\Pi_\min(e,0,0)\circ T(f_1\otimes f_2)(c).
    \end{equation*}
    Using $\Pi_\min(x)\circ d\Pi_\min(Y)=d\Pi_\min(\Ad(x)Y)\circ\Pi_\min(x)$ for $x=\exp(-tH)g^{-1}$ and $Y=(e,0,0)$ gives
    $$ = ie^{\frac{rd}{2}s}d\Pi_\min(e^{-2t}e,0,0)\circ\Pi_\min(\exp(-tH)g^{-1})\circ T(f_1\otimes f_2)(c). $$
    We can use that $T$ is intertwining for $G\times P'$ by Lemma~\ref{lem:IntertwinerQandH} to rewrite this as
    $$ = ie^{\frac{rd}{2}s}e^{-2t}d\Pi_\min(e,0,0)\circ T\Big(\ell(g^{-1})f_1\otimes \tau_{m,\nu}(\exp(-tH))f_2\Big)(c). $$
    If we now assume the claimed formula at $x=c$ and write $\xi f_2''+\frac{rd}{2}f_2'=d\tau_{m,\nu}(E)f_2$ with $\nu=\frac{rd-2}{4}$ and arbitrary $m\in\ZZ/2\ZZ$, then
    \begin{multline*}
        = e^{\frac{rd}{2}s}e^{-2t}T\Big(\ell(g^{-1})f_1\otimes\big(d\tau_{m,\nu}(E)\circ\tau_{m,\nu}(\exp(-tH))f_2\big)\Big)(c)\\
        +8e^{\frac{rd}{2}s}e^{-2t}T\Big(\big(d\ell(C)\circ\ell(g^{-1})f_1\big)\otimes \xi^{-1}\tau_{m,\nu}(\exp(-tH))f_2\Big)(c).
    \end{multline*}
    Using once more the compatibility between group and Lie algebra representation for the first term as well as \eqref{eq:GroupActionTau} for the second term, we find
    \begin{multline*}
        = e^{\frac{rd}{2}s}T\Big(\ell(g^{-1})f_1\otimes\big(\tau_{m,\nu}(\exp(-tH))\circ d\tau_{m,\nu}(E)f_2\big)\Big)(c)\\
        +8e^{\frac{rd}{2}s}T\Big(\big(\ell(g^{-1})\circ d\ell(C)f_1\big)\otimes \tau_{m,\nu}(\exp(-tH))\xi^{-1}f_2\Big)(c)
    \end{multline*}
    Another application of Lemma~\ref{lem:IntertwinerQandH} finally yields
    \begin{multline*}
        = e^{\frac{rd}{2}s}\Pi_\min(\exp(-tH)g^{-1})\circ\Big(T\big(f_1\otimes d\tau_{m,\nu}(E)f_2\big)(c)+8T\big(d\ell(C)f_1\otimes \xi^{-1}f_2\big)(c)\Big)\\
        = T\big(f_1\otimes d\tau_{m,\nu}(E)f_2\big)(e^{2t}g\cdot c)+8T\big(d\ell(C)f_1\otimes \xi^{-1}f_2\big)(e^{2t}g\cdot c).
    \end{multline*}
    This shows that if the claimed identity holds at $x=c$, then it holds on $\RR_+\calO_G$. In the non-Euclidean case, a similar computation with $-\id_V\in G'$ (using $\tau_{m,\nu}(-\id_V)\circ d\tau_{m,\nu}(E) =-d\tau_{m,\nu}(E)\circ \tau_{m,\nu}(-\id_V)$) extends this to $\RR^\times\calO_G$, which is open and dense in $\calO$ by Proposition~\ref{prop:FOHopendense}.
    
    To show the claimed identity evaluated at $x=c$, we observe that by Lemma~\ref{lem:MaxTorus} we can write $d\ell(C)f_1(c)=d\ell(C')f_1(c)$, and $d\ell(D)f_1(c)=d\ell(D')f_1(c)$ in the case $\FF=\CC$, where
    $$ C' = \sum_\alpha\varepsilon_\alpha[L(c),L(f_\alpha)]^2 \in U(\frakg) \qquad \mbox{and} \qquad D' = \sum_\alpha\varepsilon_\alpha[L(c),L(f_\alpha)]\cdot[L(c),L(if_\alpha)] $$
    with $(f_\alpha)$ a basis of $V(c,\frac{1}{2})=V(c,\frac{1}{2})^+\oplus V(c,\frac{1}{2})^-$ such that $\tau(f_\alpha,f_\beta)=\varepsilon_\alpha\delta_{\alpha\beta}$ and $\varepsilon_\alpha\in\{\pm1\}$, i.e. $\varepsilon_\alpha=\pm1$ for $f_\alpha\in V(c,\frac{1}{2})^\pm$. We start by computing the following expressions for $x=c$:
    \begin{align*}
        i\,d\Pi_\min(e,0,0) &= \sum_{\alpha,\beta}\tau(P(\widehat{e}_\alpha,\widehat{e}_\beta)x,e)\frac{\partial^2}{\partial x_\alpha\partial x_\beta}+\frac{\delta d}{2}\sum_\alpha\tau(\widehat{e}_\alpha,e)\frac{\partial}{\partial x_\alpha},\\
        i\,d\Pi_\min(ie,0,0) &= \sum_{\alpha,\beta}\tau(P(\widehat{e}_\alpha,\widehat{e}_\beta)x,ie)\frac{\partial^2}{\partial x_\alpha\partial x_\beta}+\frac{\delta d}{2}\sum_\alpha\tau(\widehat{e}_\alpha,ie)\frac{\partial}{\partial x_\alpha}.
    \end{align*}
    In the first term we can rewrite
    $$ \tau(P(\widehat{e}_\alpha,\widehat{e}_\beta)x,e) = \tau(x,P(\widehat{e}_\alpha,\widehat{e}_\beta)e) = \tau(x,\widehat{e}_\alpha\cdot\widehat{e}_\beta) $$
    and similar for $\tau(P(\widehat{e}_\alpha,\widehat{e}_\beta)x,ie)$. We complete $c_1=c$ to a Jordan frame $c_1,c_2,\ldots,c_r$ of $V^+$ and choose the basis $(e_\alpha)$ such that every $e_\alpha$ is contained in some $V_{ij}^\pm$. Then the same is true for $\widehat{e}_\alpha$. Recall that $V_{ij}\cdot V_{k\ell}=\{0\}$ for $i,j,k,\ell$ distinct, that $V_{ij}\cdot V_{jk}\subseteq V_{ik}$ for $i\neq k$, and that $V_{ij}\cdot V_{ij}\subseteq V_{ii}+V_{jj}=\FF c_i+\FF c_j$ (see \eqref{eq:VijkMultiplication} and \eqref{eq:VijMultiplication}). Moreover, $V^\pm\cdot V^\pm\subseteq V^+$ and $V^\pm\cdot V^\mp\subseteq V^-$. Hence, for $x=c=c_1$ the expression $\tau(x,\widehat{e}_\alpha\cdot\widehat{e}_\beta)$ can only be non-zero in the following cases:
    \begin{itemize}
        \item $\widehat{e}_\alpha,\widehat{e}_\beta\in V_{11}^\pm$. In the case $\FF=\RR$ we can choose $e_\alpha=e_\beta=c_1$, then $\widehat{e}_\alpha=\widehat{e}_\beta=c_1$ and hence the corresponding term in the sum becomes $\partial_{c_1}^2$. In the case $\FF=\CC$ we additionally have a contribution for $e_\alpha=e_\beta=ic_1$ and $\widehat{e}_\alpha=\widehat{e}_\beta=-ic_1$, so the term is $-\partial_{ic_1}^2$.
        \item $\widehat{e}_\alpha,\widehat{e}_\beta\in V_{1j}^\pm$. We can choose $e_\alpha=f_\alpha$, then $\widehat{e}_\alpha=\varepsilon_\alpha f_\alpha$, so that $\tau(c_1,\widehat{e}_\alpha \cdot \widehat{e}_\beta)=\varepsilon_\alpha\varepsilon_\beta \tau(c_1\cdot f_\alpha,f_\beta)=\frac{1}{2}\varepsilon_\alpha\delta_{\alpha,\beta}$. Thus, the sum over such $\alpha$ becomes $\frac{1}{2}\sum_\alpha\varepsilon_\alpha\partial_{f_\alpha}^2$.
    \end{itemize}
    In the second term we only get contributions from $\widehat{e}_\alpha=c_j$, so in total we find
    \begin{equation}
        i\,d\Pi_\min(e,0,0) = \partial_{c_1}^2-\partial_{ic_1}^2+\frac{1}{2}\sum_\alpha\varepsilon_\alpha\partial_{f_\alpha}^2+\frac{\delta d}{2}\sum_{j=1}^r\partial_{c_j}. \label{eq:BesselAtC}
    \end{equation}
    Similarly, we obtain in the case $\FF=\CC$:
    \begin{equation}
        i\,d\Pi_\min(ie,0,0) = 2\partial_{c_1}\partial_{ic_1}+\frac{1}{2}\sum_\alpha\varepsilon_\alpha\partial_{f_\alpha}\partial_{if_\alpha}+\frac{\delta d}{2}\sum_{j=1}^r\partial_{ic_j}\label{eq:BesselAtC2}
    \end{equation}

    We now compute $d\Pi_\min(C')f(c_1)$, and also $d\Pi_\min(D')f(c_1)$ in the case $\FF=\CC$. We have $[L(c_1),L(f_\alpha)]c_1=-\frac{1}{4}f_\alpha$, so we find
    $$
    d\Pi_\min(C')f(c_1)=-\frac{1}{4}\sum_\alpha\left.\varepsilon_\alpha\partial_{f_\alpha}\left(\partial_{[L(c_1),L(f_\alpha)]x}f(x)\right)\right|_{x=c_1}.
    $$
    By the product rule, this equals
    \begin{multline*}
        -\frac{1}{4}\sum_\alpha\left.\varepsilon_\alpha\Big(\partial_{[L(c_1),L(f_\alpha)]x}\partial_{f_\alpha}f(x)+\partial_{[L(c_1),L(f_\alpha)]f_\alpha}f(x)\Big)\right|_{x=c_1}\\
        = -\frac{1}{4}\sum_\alpha\varepsilon_\alpha\Big(-\frac{1}{4}\partial_{f_\alpha}^2f(c_1)+\partial_{[L(c_1),L(f_\alpha)]f_\alpha}f(c_1)\Big).
    \end{multline*}
    Since $f_\alpha^2=\frac{1}{2}\tau(f_\alpha,f_\alpha)(c_1+c_j)=\frac{1}{2}\varepsilon_\alpha(c_1+c_j)$ we find
    $$ [L(c_1),L(f_\alpha)]f_\alpha=c_1f_\alpha^2-\frac{1}{2}f_\alpha^2=\frac{1}{4}\varepsilon_\alpha(c_1-c_j) $$
    and the action of $d\Pi_\min(C)$ at $x=c_1$ becomes
    \begin{equation}
        d\Pi_\min(C')f(c_1) = \frac{1}{16}\sum_\alpha\varepsilon_\alpha\partial_{f_\alpha}^2f(c_1)-\frac{\delta d}{16}\sum_{j=2}^r\partial_{c_1-c_j}f(c_1).\label{eq:CasimirAtC}
    \end{equation}
    Putting \eqref{eq:BesselAtC} and \eqref{eq:CasimirAtC} together we conclude that
    $$ i\,d\Pi_\min(e,0,0)f(c_1) = \partial_{c_1}^2f(c_1)-\partial_{ic_1}^2f(c_1) +\delta\frac{rd}{2}\partial_{c_1}f(c_1) + 8d\Pi_\min(C)f(c_1). $$
    We can rewrite $\partial_{c_1}T(f_1\otimes f_2)=T(f_1\otimes\partial_\xi f_2)$ and $\partial_{ic_1}T(f_1\otimes f_2)=T(f_1\otimes\partial_\eta f_2)$ as well as $d\Pi_\min(C)T(f_1\otimes f_2)=T(d\ell(C)f_1\otimes f_2)$ by Lemma~\ref{lem:IntertwinerQandH}, so this is the claimed formula at $x=c_1$. Similarly, we find that
    $$ d\Pi_\min(D')f(c_1) = \frac{1}{16}\sum_\alpha\varepsilon_\alpha\partial_{f_\alpha}\partial_{if_\alpha}f(c_1)-\frac{\delta d}{16}\sum_{j=2}^r\partial_{i(c_1-c_j)}f(c_1), $$
    and together with \eqref{eq:BesselAtC2} this gives
    \begin{equation*}
        i\,d\Pi_\min(ie,0,0)f(c_1) = 2\partial_{c_1}\partial_{ic_1}f(c_1)+\delta\frac{rd}{2}\partial_{ic_1}f(c_1)+8d\Pi_\min(D')f(c_1),
    \end{equation*}
    which is the desired formula at $x=c_1$.
\end{proof}

\begin{theorem}\label{thm:AbstractTheta}
    If
    $$ L^2(\calO_G) \simeq \int^\oplus_{\widehat{G}} \pi\,d\mu(\pi) $$
    is the decomposition of the left-regular representation $\ell$ of $G$ on $L^2(\calO_G)$, then the restriction of $\Pi_\min$ to $G\times\widetilde{G'}$ is given by
    $$ \Pi_\min|_{G\times\widetilde{G'}} \simeq \int^\oplus_{\widehat{G}} \pi\boxtimes\theta(\pi)\,d\mu(\pi), $$
    where $\theta(\pi)$ is an irreducible unitary representation of $\widetilde{G'}$, a finite cover of $G'$, on $L^2(\RR_+,\xi^{\frac{rd}{2}}\,d^\times\xi)$ (for $V$ Euclidean) resp. $L^2(\FF^\times,|\zeta|^{\delta\frac{rd}{2}}\,d^\times\zeta)$ (for $V$ non-Euclidean or complex) that agrees with $\tau_{m,\nu}$ on $P'$ and whose derived representation of $\sl(2,\FF)$ agrees on the subspace $C_c^\infty(\RR_+)$ resp. $C_c^\infty(\FF^\times)$ with $d\tau_{m,\nu}$, where $m$ and $\nu$ are related to the eigenvalues $d\pi(C)$ and, in case $\FF=\CC$, also $d\pi(D)$ in the following way:
    \begin{itemize}
        \item If $\FF=\RR$ then $d\pi(C)=-\frac{1}{32}(4\nu^2-(\frac{rd}{2}-1)^2)$.
        \item If $\FF=\CC$ then $d\pi(C)=-\frac{1}{8}(\nu^2+m^2-(\frac{rd}{2}-1)^2)$ and $d\pi(D)=-\frac{1}{4}im\nu$.
    \end{itemize}
\end{theorem}

\begin{proof}
    It follows from the isomorphism \eqref{eq:IntertwinerIso} and Lemma~\ref{lem:IntertwinerQandH} that $\Pi_\min|_{G\times\widetilde{G'}}$ decomposes into $\pi\boxtimes\theta(\pi)$ with $\theta(\pi)$ a representation of $\widetilde{G'}$ on $L^2(\RR_+,|\xi|^{\frac{rd}{2}}\,d^\times\xi)$ resp. $L^2(\FF^\times,|\zeta|^{\delta\frac{rd}{2}}\,d^\times\zeta)$ on which $P'$ acts by $\tau_{m,\nu}|_{P'}$. The remaining part follows by comparing the formulas in Lemma~\ref{lem:KeyLemma} to \eqref{eq:SL2Bessel}, \eqref{eq:SL2Bessel1} and \eqref{eq:SL2Bessel2}.
\end{proof}

\begin{remark}\label{rem:AbstractTheta}
If $V$ is either Euclidean or complex, it turns out that the properties in Theorem~\ref{thm:AbstractTheta} characterize $\theta(\pi)$ uniquely, and this is made explicit in the next section. Unfortunately, for $V$ non-Euclidean the two inequivalent representations $\tau_{m,\nu}$ ($m\in\ZZ/2\ZZ$, $\nu\in i\RR$) of $G'=\PGL(2,\RR)$ both have the same restriction to $P'$ and the same infinitesimal action on $C_c^\infty(\RR^\times)$, so the information in Theorem~\ref{thm:AbstractTheta} is not sufficient to determine the theta lifts of all representations $\pi$. In the next section, we will consider some small $K$-types to solve this problem.
\end{remark}

\section{The Plancherel formula for \texorpdfstring{$\calO_G$}{OH}}\label{sec:ExplicitTheta}

In this section we use results from the literature to make the Plancherel formula for $\calO_G$ explicit in all cases. As a consequence, the theta correspondence from Theorem~\ref{thm:AbstractTheta} becomes completely explicit and is shown to be one-to-one.

\subsection{$V$ Euclidean}

If $V$ is Euclidean, $G$ is compact and the decomposition of $L^2(\calO_G)=L^2(G/G_c)$ is discrete and given by the Cartan--Helgason Theorem. For this, recall from Lemma~\ref{lem:MaxTorus} that $\frakt=\RR T_0$ is a maximal torus in $\frakg^{-\sigma}$. The Cartan--Helgason Theorem~\cite[Ch. V, Theorem 4.1]{Hel84} asserts that $L^2(\calO_G)$ decomposes into the multiplicity-free direct sum over all irreducible representations of $G$ of highest weight $\lambda\in\frakt_\CC^*$ such that
\begin{itemize}
    \item $\frac{\langle\lambda,\mu\rangle}{\langle\mu,\mu\rangle}\in\ZZ_{\geq0}$ for all $\mu\in\Sigma^+(\frakg_\CC,\frakt_\CC)$,
    \item $e^{\lambda}$ is trivial on $\exp(\frakt)\cap G_c$.
\end{itemize}
In view of Proposition~\ref{prop:RootSystem}~\eqref{prop:RootSystem2}, the first condition leads to $\lambda=k\beta$ such that $k\in\ZZ_{\geq0}$ if $d=1$ and $k\in2\ZZ_{\geq0}$ if $d>1$. But $e^{tT_0}\in G_c$ if and only if $t\in\pi\ZZ$ (see Lemma~\ref{lem:TorusStabilizer}), so the second condition forces $k\in2\ZZ_{\geq0}$ in all cases. This implies
$$ L^2(\calO_G) \simeq \widehat{\bigoplus_{k\in2\ZZ_{\geq0}}} \,\pi_{k\beta}, $$
where $\pi_{k\beta}$ is the irreducible representation of $G$ with highest weight $k\beta$.

Since $\rho_\frakt=(\frac{rd}{2}-1)\beta$ (see Proposition~\ref{prop:RootSystem}~\eqref{prop:RootSystem2}), the eigenvalue $d\pi_{k\beta}(C)$ of the Casimir element $C$ on $\pi_{k\beta}$ is given by
$$ d\pi_{k\beta}(C) = B(k\beta+2\rho_\frakt,k\beta) = -\frac{1}{32}k(k+rd-2), $$
where we have used that under the identification $\frakt\simeq\frakt^*$ in terms of the Killing form $B$, we have $B(\beta,\beta)=-\frac{1}{32}$.

Theorem~\ref{thm:AbstractTheta} implies that $\theta(\pi_{k\beta})$ is an irreducible unitary representation of a finite cover of $G'=\PSL(2,\RR)$ on $L^2(\RR_+,|\xi|^{\frac{rd}{2}}d^\times\xi)$ that agrees with $\tau^{\textup{hds}}_{k+\frac{rd}{2}}$ on $P'$ and whose derived representation agrees on $C_c^\infty(\RR_+)$ with $d\tau^{\textup{hds}}_{k+\frac{rd}{2}}$. In view of Theorem~\ref{thm:UnitaryDualPGL} this implies:

\begin{corollary}\label{cor:ExplicitThetaEuclidean}
    $\theta(\pi_{k\beta})=\tau^{\textup{hds}}_{k+\frac{rd}{2}}$ is the holomorphic discrete series representation of $\widetilde{\PSL}(2,\RR)$ with parameter $k+\frac{rd}{2}$.
\end{corollary}

\subsection{$V$ non-Euclidean}\label{sec:PlancherelNonEuclidean}

Assume that $V$ is real and non-Euclidean, and that $V\not\simeq\RR^{p,q}$ with $p+q$ odd. In this case $\calO_G=G/G_c$ is a non-Riemannian semisimple symmetric space with $\rank(G/G_c)=1$ by Lemma~\ref{lem:MaxTorus}, so its Plancherel formula is known by e.g. \cite{Mol92,vdB05}. It consists of a continuous and a discrete part which we treat separately.

For the continuous part, recall from Lemma~\ref{lem:MaxTorus} the maximal torus $\fraka=\RR H_0\subseteq\frakg^{-\sigma}\cap\frakg^{-\theta}$ with root system $\Sigma(\frakg,\fraka)$ of the form $\{\pm\alpha\}$, $\{\pm2\alpha\}$ or $\{\pm\alpha,\pm2\alpha\}$. Then the subalgebra $\frakp=\frakg^0\oplus\frakg^{\alpha}\oplus\frakg^{2\alpha}$ of $\frakg$ is parabolic, and we let $P=N_G(\frakp)=MAN$ denote the corresponding parabolic subgroup of $G$. The intersection $M\cap G_c$ is equal to the stabilizer $M_c$ of $c$ in $M$. Recall that we extend $c_1=c$ to a Jordan frame $c_1,c_2,\ldots,c_r$.

\begin{lemma}\label{lem:MHquotientForNonEuclidean}
    $M/M_c=\{1,m_0\}$, where $m_0\in M$ is an element with $m_0c_1=c_2$, $m_0c_2=c_1$ and $m_0h_0=-h_0$.
\end{lemma}

\begin{proof}
    We claim that $M\cdot(c_1-c_2)=\{\pm(c_1-c_2)\}$ and that $M_c=M_{c_1}$ is the stabilizer of $c_1-c_2$ in $M$. This will imply that $M/M_c\simeq M\cdot(c_1-c_2)=\{\pm(c_1-c_2)\}$.\\
    To this end, we first observe that $M\cdot(c_1-c_2)$ must be finite. In fact, from Proposition~\ref{prop:RootSystem}~\eqref{prop:RootSystem1} we know that
    $$ \frakm = \frakg_{0,h_0}\oplus\bigoplus_{3\leq i<j\leq r}\frakg_{ij}, $$
    where $\frakg_{0,h_0}$ denotes the stabilizer of $h_0$ in $\frakg_0$ and $\frakg_{ij}=[L(c_i),L(V_{ij})]$. This implies $\frakm\cdot(c_1-c_2)=\{0\}$. Since $M$ has finitely many connected components, it follows that $M\cdot(c_1-c_2)$ is finite.\\
    Next we argue that $M\cdot(c_1-c_2)\subseteq\{\pm(c_1-c_2)\}$. For this, we first observe (in a similar way to Lemma~\ref{lem:TorusStabilizer}) that the one parameter group $A=e^{\RR H_0}$ is acting on $V$ by
    \begin{align*}
        e^{tH_0}(c_1-c_2) &= \cosh(2t)(c_1-c_2)-\frac{1}{4}\sinh(2t)h_0,\\
        e^{tH_0}(c_1+c_2) &= c_1+c_2,\\
        e^{tH_0}h_0 &= \cosh(2t)h_0-4\sinh(2t)(c_1-c_2),\\
        e^{tH_0}x &= x && (x\in V_{12},\tau(x,h_0)=0),\\
        e^{tH_0}y &= \cosh(t)y+\sinh(t)z && (y\in V_{1j},z=-\frac{1}{2}L(h_0)y\in V_{2j}),\\
        e^{tH_0}z &= \cosh(t)z+\sinh(t)y && (z\in V_{2j},y=\frac{1}{2}L(h_0)z\in V_{1j}),\\
        e^{tH_0}w &= w && (w\in \bigoplus_{3\leq i\leq j\leq r}V_{ij}).
    \end{align*}
    So if $g\in M$, then $g$ centralizes $H_0$ and hence $g(c_1-c_2)$ must satisfy
    $$ e^{tH_0}g(c_1-c_2) = ge^{tH_0}(c_1-c_2) = \cosh(2t)g(c_1-c_2)-\frac{1}{4}\sinh(2t)gh_0. $$
    A similar expression holds for $gh_0$, so we conclude that $u=g(c_1-c_2)$ and $v=gh_0$ are solutions of the system
    \begin{align}
        e^{tH_0}u &= \cosh(2t)u-\frac{1}{4}\sinh(2t)v,\\
        e^{tH_0}v &= -4\sinh(2t)u+\cosh(2t)v.
    \end{align}
    In view of the action of $A=\exp(\RR H_0)$ on $V$ determined above, this shows that $u,v\in\RR(c_1-c_2)+\RR h_0$, more precisely $u=a(c_1-c_2)+bh_0$ and $v=16b(c_1-c_2)+ah_0$ for some $a,b\in\RR$. In particular, $g(c_1-c_2)=a(c_1-c_2)+bh_0$. Moreover, $g$ preserves the trace form $\tau$, which is on $\RR(c_1-c_2)+\RR h_0$ given by
    $$ \tau(a_1(c_1-c_2)+b_1h_0,a_2(c_1-c_2)+b_2h_0) = 2a_1a_2-32b_1b_2, $$
    so we know that $a^2-16b^2=1$. Writing $(a,b)=\pm(\cosh(2s),\frac{1}{4}\sinh(2s))$ for some $s\in\RR$ gives $g(c_1-c_2)=\pm e^{sH_0}(c_1-c_2)$ and hence
    $$ g^n(c_1-c_2)=(-1)^n e^{snH_0}(c_1-c_2) = (-1)^n\left(\cosh(2ns)(c_1-c_2)-\frac{1}{4}\sinh(2ns)h_0\right) \qquad (n\in\ZZ). $$
    For $\{g^n(c_1-c_2):n\in\ZZ\}$ to be finite, we clearly need $s=0$, so $g(c_1-c_2)=\pm(c_1-c_2)$.\\
    Now, this also implies that if $g\in M$ stabilizes $c_1$, then it also stabilizes $c_1-c_2$. Conversely, if $g\in M$ stabilizes $c_1-c_2$, then $gh_0= h_0$ by the same argument as above. As shown in the proof of Lemma~\ref{lem:D0commutators}, we have $h_0^2=-16(c_1+c_2)$, so $g(c_1+c_2)=-\frac{1}{16}g(h_0^2)=-\frac{1}{16}(gh_0)^2=-\frac{1}{16}h_0^2=c_1+c_2$ and hence $g$ also stabilizes $c_1+c_2$, whence $c_1$. This shows that $M_c$ is the stabilizer of $c_1-c_2$ in $M$.\\
    It remains to construct $m_0\in M$ with $m_0c_1=c_2$, $m_0c_2=c_1$ and $m_0h_0=-h_0$. Note that any $m_0\in G$ with these properties is automatically contained in $M$. We define $m_0$ to be the product (in whichever order) of the elements $\exp(\pi(c_1,0,-c_1))$ and $\exp(\frac{\pi}{2}(0,H_0,0))$. The latter element is contained in $G$ since $H_0\in\frakg$, and it maps $c_1\mapsto c_2$, $c_2\mapsto c_1$ and $h_0\mapsto h_0$. The former element is contained in $G$ and acts on $V(c_1,1)\oplus V(c_1,0)$ by $+1$ and on $V(c_1,\frac{1}{2})$ by $-1$ due to Lemma~\ref{lem:OHisSymmetricSpace}. It therefore maps $c_1\mapsto c_1$, $c_2\mapsto c_2$ and $h_0\mapsto-h_0$, so the claim follows.
\end{proof}

Identifying $\widehat{M/M_c}\simeq\xi\in\ZZ/2\ZZ$ we define for $\xi\in\ZZ/2\ZZ$ and $\lambda\in i\fraka^*$ the unitary principal series representation $\pi_{\xi,\lambda}$ by (normalized parabolic induction)
$$ \pi_{\xi,\lambda} = \Ind_{P}^G(\xi\otimes e^\lambda\otimes1). $$
We write $i\fraka_+^*$ for all $\lambda=\mu\alpha$ with $\mu\in i\RR_+$.

To describe the discrete part of the Plancherel formula, recall from Lemma~\ref{lem:MaxTorus} the maximal torus $\frakt=\RR T_0\subseteq\frakg^{-\sigma}\cap\frakg^\theta$ as well as the corresponding root system $\Sigma(\frakg_\CC,\frakt_\CC)$ that is expressed using $\beta\in\frakt_\CC^*$. The subalgebra $\frakq=\frakg_{\CC}^{0}\oplus\frakg_{\CC}^{\beta}\oplus\frakg_{\CC}^{2\beta}$ of $\frakg_\CC$ is parabolic and $\theta$-stable. For $\lambda\in\frakt_\CC^*$ we let $A_{\mathfrak{q}}(\lambda)$ denote Zuckerman's derived functor module.

The Plancherel formula for $\calO_G$ can now be written as (see e.g. \cite[equation (9.3)]{Mol92} or \cite[Theorem 10.21]{vdB05})
$$ L^2(\calO_G) \simeq \bigoplus_{\xi\in\ZZ/2\ZZ}\int^\oplus_{i\mathfrak{a}^*_+}\pi_{\xi,\lambda}\,d\lambda \oplus \widehat{\bigoplus_\lambda} A_{\mathfrak{q}}(\lambda). $$
The direct sum is over all discrete series parameters $\lambda$ which we now make explicit:

\begin{lemma}\label{lem:DSparameters}
    The discrete series parameters for $G/G_c$ are $\lambda=k\beta$ with $k+\frac{rd}{2}\in2\ZZ_{>0}$.
\end{lemma}

\begin{proof}
    By \cite{Mat88,OM84} the discrete series parameters are all $\lambda\in\frakt_\CC^*$ such that
    \begin{itemize}
        \item (fair range) $\langle\lambda+\rho_\frakt,\gamma\rangle>0$ for all positive roots $\gamma$,
        \item $\frac{\langle\mu,\gamma\rangle}{\langle\gamma,\gamma\rangle}\in\ZZ$ for all positive roots $\gamma$, where $\mu=\lambda+2\rho_\frakt-2\rho_{\frakt,\textup{cpt}}$, where $2\rho_{\frakt,\textup{cpt}}$ is the sum of all positive roots of $(\frakk_\CC,\frakt_\CC)$,
        \item $e^\mu$ is trivial on $\exp(\frakt)\cap G_c$.
    \end{itemize}
    Since $\rho_\frakt=(\frac{rd}{2}-1)\beta$ by Proposition~\ref{prop:RootSystem}~\eqref{prop:RootSystem2}, the first condition becomes $k+\frac{rd}{2}>1$. (Here the form $\langle\cdot,\cdot\rangle$ is normalized such that $\langle\beta,\beta\rangle>0$.) The second condition is sufficient for $\gamma=2\beta$ which is always a root (in contrast to $\beta$ itself). Since $\rho_\frakt=(\frac{rd}{2}-1)\beta$ and $\rho_{\frakt,\textup{cpt}}=(\frac{rd_+}{2}-1)\beta$ with $d_+=\dim V_{ij}^+$ ($i<j$) this condition becomes $k+r(d-d_+)\in2\ZZ$. For $r=2$ we obtain $k\in2\ZZ$, which is equivalent to $k+\frac{rd}{2}\in2\ZZ$ since $d=p+q-2$ is even. For $r>2$ we always have $d=2d_+$, so the condition is $k+\frac{rd}{2}\in2\ZZ$. Under these assumptions, the third condition is always satisfied by Lemma~\ref{lem:TorusStabilizer}.
    %We use the formulation in \cite[Theorem 2.9]{Vog88} (note that $A_\frakq(\lambda)$ is $A(X)$ with $X=\CC_{\lambda+\rho_\frakt}$ in the notation of \cite{Vog88}). The discrete series parameters are all $\lambda\in\frakt_\CC^*$ in the fair range which integrate to a character of the corresponding Levi subgroup of $H$ and are trivial on its intersection with $G^\sigma$ (see e.g. \cite[Theorem 2.9]{Vog88}).
    %For $\lambda=k\beta$ to be in the fair range, we need $\Re\langle\lambda+\rho_\frakt,\beta\rangle>0$. Since $\rho_t=(\frac{rd}{2}-1)\beta$ by Proposition~\ref{prop:RootSystem}~\eqref{prop:RootSystem2}, this condition becomes $\Re(k+\frac{rd}{2}-1)>0$. (Here we normalize the form $\langle\cdot,\cdot\rangle$ such that $\langle\beta,\beta\rangle>0$.)
    %Next, we determine those $\lambda=k\beta$ that lift to a character of the Levi subgroup. By Proposition~\ref{prop:RootSystem}~\eqref{prop:RootSystem2} we know that $\exp(tT_0)=1$ if and only if $t\in2\pi\ZZ$ for $r>2$, and $t\in\pi\ZZ$ for $r=2$. Hence, $k\in\ZZ$ for $r>2$ and $k\in2\ZZ$ for $r=2$.
    %Finally, $\exp(\frakt)\cap G^\sigma=\exp(\ZZ T_0)$, so the condition that $e^\lambda$ is trivial on $\exp(\frakt)\cap G^\sigma$ is equivalent to $k\in2\ZZ$.
\end{proof}

We first compute the eigenvalue of the Casimir element $C\in U(\frakg)$ in the representations $\pi$ occurring in the Plancherel formula for $\calO_G$.

\begin{lemma}\label{lem:CasimirEigenvalues}
    Let $C\in U(\frakg)$ be the Casimir element with respect to the invariant bilinear form $B$ (see \eqref{eq:NormalizationKillingForm}).
    \begin{enumerate}[(1)]
        \item\label{lem:CasimirEigenvaluesInd} For $\pi=\pi_{\xi,\lambda}$ with $\xi\in\ZZ/2\ZZ$ and $\mu=\lambda(H_0)\in\CC\simeq\fraka_\CC^*$ the Casimir eigenvalue is given by
        $$ d\pi(C) = -\frac{1}{32}\left(\mu+\frac{rd}{2}-1\right)\left(\mu-\frac{rd}{2}+1\right). $$
        \item\label{lem:CasimirEigenvaluesAqL} For $\pi=A_{\frakq}(\lambda)$ with $\lambda=k\beta$, the Casimir eigenvalue is given by
        $$ d\pi(C) = -\frac{1}{32}k(k+rd-2). $$
    \end{enumerate}
\end{lemma}

\begin{proof}
    We identify $\fraka_\CC\simeq\fraka_\CC^*$ and $\frakt_\CC\simeq\frakt_\CC^*$ via the form $B$ and use these identifications to define invariant bilinears form on $\fraka_\CC^*$ and $\frakt_\CC^*$, also denoted by $B$. By \eqref{eq:NormalizationKillingForm} these forms satisfy $B(\alpha,\alpha)=B(\beta,\beta)=-\frac{1}{32}$. By Lemma~\ref{prop:RootSystem} we have $\rho_\fraka=(\frac{rd}{2}-1)\alpha$ and $\rho_\frakt=(\frac{rd}{2}-1)\beta$. Then \eqref{lem:CasimirEigenvaluesInd} follows from \cite[Proposition 8.22]{Kna86} and \eqref{lem:CasimirEigenvaluesAqL} follows from \cite[Corollary 5.25]{KV95}, together with the fact that the Casimir element acts on an irreducible representation with infinitesimal character $\lambda$ by $B(\lambda+\rho,\lambda-\rho)$.
\end{proof}

Unfortunately, the Casimir eigenvalues of $\pi_{\xi,\lambda}$ for the two different $\xi\in\ZZ/2\ZZ$ are equal, so this information is not sufficient to determine $\theta(\pi_{\xi,\lambda})$. We therefore consider the restriction of some small $K$-types. For this, let $K=G^\theta$ denote the maximal compact subgroup of $G$ corresponding to the Cartan involution $\theta$, and $K'=\operatorname{PO}(2)=\operatorname{O}(2)/\{\pm I\}\subseteq G'=\PGL(2,\RR)$.

\begin{lemma}\label{lem:KtypeAnalysisNonEuclidean}
    The restriction of $\Pi_\min$ to $K\times K'$ contains representations of the form $W_0\boxtimes\CC_{\triv}$ and $W_1\boxtimes\CC_{\sgn}$, where $\CC_{\triv}$ resp. $\CC_{\sgn}$ denotes the trivial resp. non-trivial character of $K'=\operatorname{PO}(2)$ and $W_\xi$ is a $K$-type of $\pi_{\xi,\lambda}$ and not of $\pi_{\xi+1,\lambda}$ ($\xi\in\ZZ/2\ZZ$), except in the case $V\simeq\RR^{p,q}$ with $p-q\equiv2\pmod4$ where $W_\xi$ is a $K$-type of $\pi_{\xi+1,\lambda}$ and not of $\pi_{\xi,\lambda}$ ($\xi\in\ZZ/2\ZZ$).
\end{lemma}

\begin{proof}
    We first show the statement for $r=2$, i.e. $\calG=\SO(p+1,q+1)$ with $G=\SO(p-1,q)$ and $G'\simeq\SO(2,1)$ with $p+q$ even. Then $K=\textup{S}(\upO(p-1)\times\upO(q))$ and $K'\simeq\textup{S}(\upO(2)\times\upO(1))\simeq\upO(2)$. Let us assume that $p\geq q$, then the $K$-types of $\Pi_\min$ are given by $V_k=\calH^k(\RR^{p+1})\boxtimes\calH^{k+\frac{p-q}{2}}(\RR^{q+1})$, $k\in\ZZ_{\geq0}$, where we write $\calH^k(\RR^n)$ for the irreducible representation of $\upO(n)$ or $\SO(n)$ on the space of harmonic homogeneous polynomials on $\RR^n$ of degree $k$. They satisfy the following classical branching law, which can be extracted from \cite[Appendix 7.5]{KobaMano11} together with classical results on Gegenbauer polynomials:
    $$ \calH^k(\RR^n)|_{\upO(n-1)\times\upO(1)} \simeq \bigoplus_{\ell=0}^k \calH^\ell(\RR^{n-1})\boxtimes\CC_{\sgn}^{k-\ell}, $$
    where use the notation $\CC_{\sgn}^m$ for $\CC_{\triv}$ if $m$ is even and $\CC_{\sgn}$ if $m$ is odd. Restricting the $K$-types for $k=0$ and $k=1$ to $K\times K'$ gives 
    \begin{align*}
        V_0|_{K\times K'} &\simeq \bigoplus_{\ell=0}^{\frac{p-q}{2}}\big(\calH^0(\RR^{p-1})\boxtimes\calH^\ell(\RR^q)\big)\boxtimes\CC_{\sgn}^{\frac{p-q}{2}-\ell},\\
        V_1|_{K\times K'} &\simeq \bigoplus_{\ell=0}^{\frac{p-q}{2}+1}\Big[\big(\calH^1(\RR^{p-1})\boxtimes\calH^\ell(\RR^q)\big)\boxtimes\CC_{\sgn}^{\frac{p-q}{2}+1-\ell}\Big]\oplus\Big[\big(\calH^0(\RR^{p-1})\boxtimes\calH^\ell(\RR^q)\big)\boxtimes\calH^1(\RR^2)\Big].
    \end{align*}
    It is known (see \cite[Lemma 2.2 and the discussion thereafter]{HT93}) that $\calH^k(\RR^{p-1})\boxtimes\calH^\ell(\RR^q)$ is contained in $\pi_{\xi,\lambda}$ iff $k+\ell\equiv\xi\pmod2$. We can therefore choose $W_0=\calH^0(\RR^{p-1})\boxtimes\calH^{\frac{p-q}{2}}(\RR^q)$ and $W_1=\calH^1(\RR^{p-1})\boxtimes\calH^{\frac{p-q}{2}}(\RR^q)$. %If $\frac{p-q}{2}$ is even, we can therefore choose $W_0=\calH^0(\RR^{p-1})\boxtimes\calH^{0}(\RR^q)$ and $W_1=\calH^1(\RR^{p-1})\boxtimes\calH^{\O}(\RR^q)$. If $\frac{p-q}{2}$ is odd, we exchange the spaces and choose $W_0=\calH^1(\RR^{p-1})\boxtimes\calH^{0}(\RR^q)$ and $W_1=\calH^0(\RR^{p-1})\boxtimes\calH^{0}(\RR^q)$.    
    The argument for $p<q$ is similar, in this case we can choose $W_0=\calH^{\frac{q-p}{2}}(\RR^{p-1})\boxtimes\calH^0(\RR^q)$ and $W_1=\calH^{\frac{q-p}{2}+1}(\RR^{p-1})\boxtimes\calH^0(\RR^q)$.

    Now let us assume that $r>2$, then the $K$-types of $\Pi_\min$ have highest weight $k\beta$ ($k\in\ZZ_{\geq0}$), where $\beta$ is the highest weight of $\frakp_\CC$ (see e.g. \cite[Theorem 3.8]{HKM14}). In particular, the trivial representation $\CC_{\triv}$ and the representation $\frakp_\CC$ are $K$-types of $\Pi_\min$.
    
    Restricting the trivial representation to $K\times K'$ clearly yields $\CC_{\triv}\boxtimes\CC_{\triv}$, and the trivial representation only occurs in those parabolically induced representations whose induction data for $M$ are trivial, i.e. it occurs in $\pi_{0,\lambda}$ and not in $\pi_{1,\lambda}$. We can therefore choose $W_0=\CC_{\triv}$.

    We now restrict $\frakp_\CC$ to $K\times K'$. Recall from \eqref{eq:LieAlgP} that
    $$ \frakp_\CC = \{(u,T,\vartheta(u)):u\in V_\CC,T=T^*\}. $$
    Then $K'=\operatorname{PO}(2)$ acts on the subspace $\{(u,0,-u):u\in V^-\}$ by the sign character and $K\subseteq G\subseteq\Aut(V)$ acts by its restriction to $V^-$. So $\frakp_\CC|_{K\times K'}$ contains $V_\CC^-\boxtimes\CC_{\sgn}$. It remains to show that $W_1=V^-_\CC$ is a $K$-type of $\pi_{1,\lambda}$, but not of $\pi_{0,\lambda}$. By Frobenius reciprocity this is equivalent to the statement that the representation of $M\cap K$ on $V^-_\CC$ contains the non-trivial character of $M\cap K$, but not the trivial representation. It clearly contains the non-trivial character as $\CC h_0$, so we have to show that it does not contain the trivial representation.

    Note that for every $1\leq j\leq r$ the element $m_j=\exp(\pi(c_j,0,-c_j))$ is contained in $G$ and acts on $V(c_j,1)\oplus V(c_j,0)$ by $+1$ and on $V(c_j,\frac{1}{2})$ by $-1$ (see Lemma~\ref{lem:OHisSymmetricSpace}). For $j\geq3$ it commutes with $H_0$, so it is contained in $M$. Since $V_{ij}\subseteq V(c_j,\frac{1}{2})$ for every $i\neq j$, it follows that any copy of the trivial representation in $V^-=\bigoplus_{1\leq i<j\leq r}V_{ij}^-$ has to be contained in $V_{12}^-$. But the element $m_0\in M\cap K$ constructed in the proof of Lemma~\ref{lem:MHquotientForNonEuclidean} acts on $V_{12}^-$ by $-1$. So $V^-$ does not contain a copy of the trivial representation of $M\cap K$.
\end{proof}

\begin{corollary}\label{cor:ExplicitThetaNonEuclidean}
    \begin{enumerate}[(a)]
        \item For $\xi\in\ZZ/2\ZZ$ and $\lambda\in i\fraka^*$ we have $\theta(\pi_{\xi,\lambda})=\tau_{\xi,\nu}$, where $\nu=\frac{1}{2}\lambda(H_0)\in i\RR$, except for $V\simeq\RR^{p,q}$ with $p-q\equiv2\pmod4$ where $\theta(\pi_{\xi,\lambda})=\tau_{\xi+1,\nu}$.
        \item For $\lambda=k\beta$, $k\in\ZZ$, $k>-(\frac{rd}{2}-1)$, we have $\theta(A_{\frakq}(\lambda))=\tau_{k+\frac{rd}{2}}^{\textup{ds}}$.
    \end{enumerate}
\end{corollary}

\begin{proof}
    We first show (b). By Theorem~\ref{thm:AbstractTheta} and Lemma~\ref{lem:CasimirEigenvalues}, the unitary representation $\theta(A_{\frakq}(\lambda))$ of $G'$ has an irreducible infinite-dimensional restriction to $P'$, and its Lie algebra action on $C_c^\infty(\RR^\times)$ is the same as the one of the discrete series representation $\tau^{\textup{ds}}_{k+\frac{rd}{2}}$. This implies that the irreducible unitary representations $\theta(A_{\frakq}(\lambda))$ and $\tau^{\textup{ds}}_{k+\frac{rd}{2}}$ have the same infinitesimal character, so by the classification of the unitary dual of $\PGL(2,\RR)$, it follows that they are equivalent.

    To show (a) we first observe that by the same argument as above, $\theta(\pi_{\xi,\lambda})$ is isomorphic to $\tau_{\eta,\nu}$ for some $\eta\in\ZZ/2\ZZ$. We claim that $\eta=\xi$ (resp. $\eta=\xi+1$ in the case $V\simeq\RR^{p,q}$ with $p-q\equiv2\pmod4$). Note that by Lemma~\ref{lem:KtypeAnalysisNonEuclidean} there is a $(K\times K')$-representation $W$ in $\Pi_\min|_{K\times K'}$ which is a $(K\times K')$-type of $\pi_{\xi,\lambda}\boxtimes\tau_{\eta,\nu}$ for $\eta=\xi$ (resp. $\eta\neq\xi$), but not for $\eta\neq\xi$ (resp. $\eta=\xi$). Since this $K$-type does not occur in the discrete spectrum (the discrete series for $G'=\PGL(2,\RR)$ does never contains $\CC_{\triv}$ or $\CC_{\sgn}$ as $K'$-type), it has to occur for some part of the continuous spectrum. So for a subset of $\lambda\in i\fraka^*$ of positive measure, we must have $\eta=\xi$ (resp. $\eta=\xi+1$). To show that this implies $\eta=\xi$ (resp. $\eta=\xi+1$) for almost all $\lambda\in i\fraka^*$, we use an analyticity argument.

    By \cite{Fra22}, the projections to the representations in the direct integral are given in terms of a family of intertwining operators
    $$ L^2(\calO)^\infty \to \pi_{\xi,\lambda}\boxtimes L^2(\RR^\times,|\zeta|^{\frac{rd}{2}-1}\,d^\times\zeta), $$
    and by \cite{Mol92} these intertwining operators can be written as (after possibly renormalizing them)
    \begin{equation}
        f\mapsto\left[(h,\xi)\mapsto\int_{\calO_G}u_{\xi,\lambda}(x)f(\xi h^{-1}x)\,d\mu_G(x)\right],\label{eq:FamilyOfIntertwiners}
    \end{equation}
    where $u_{\xi,\lambda}$ is a family of spherical distributions on $\calO_G$ depending holomorphically on $\lambda\in i\fraka^*$. Now let $W\subseteq L^2(\calO)^\infty$ be the $(K\times K')$-type from above, embedded into $L^2(\calO)$. By the argument above, the image of $W$ under \eqref{eq:FamilyOfIntertwiners} is non-zero for some set of $\lambda\in i\fraka^*$ of positive measure. This implies that the above family of intertwining operators cannot be constant equal to zero. But since it depends holomorphically on $\lambda$, the image of $W$ is non-trivial for almost all $\lambda\in i\fraka^*$. This shows that the trivial representation resp. the sign character is a $K'$-type in $\tau_{\eta,\nu}$ ($\nu=\frac{1}{2}\lambda(H_0)$) for almost all $\lambda$. But this can only be the case if $\eta=\xi$ (resp. $\eta=\xi+1$).
\end{proof}

\subsection{$V$ complex}

Now assume that $V$ is complex, then $\calO_G=G/G_c$ is a complex symmetric space and its Plancherel formula is purely continuous.

Recall the maximal split torus $\fraka=\RR H_0$ of $\frakg^{-\sigma}\cap\frakg^{-\theta}$, where $H_0=[L(c_1),L(h_0)]$, and form the parabolic subgroup $P$ with Lie algebra $\frakp=\frakg^0\oplus\frakg^\alpha\oplus\frakg^{2\alpha}$. Write $P=MAN$ for the according Langlands decomposition with $A=\exp\fraka$. Then $\frakm$ has the one-dimensional compact torus $\RR(iH_0)$ as ideal. More precisely:

\begin{lemma}
    The map
    $$ M/M_c\to\{\cos(t)(c_1-c_2)+\tfrac{1}{4}\sin(t)h_0:t\in\RR\}, \quad mM_c\mapsto m\cdot(c_1-c_2), $$
    is a diffeomorphism. In particular, we have an isomorphism of Lie groups
    $$ \upU(1)\to M/M_c, \quad e^{2it}\mapsto e^{itH_0}(M_c). $$
\end{lemma}

\begin{proof}
    The proof is similar to the proof of Lemma~\ref{lem:MHquotientForNonEuclidean}, so we omit it.
\end{proof}

For $k\in\ZZ$ we consider the character
$$ \chi_k:M\to\CC^\times, \quad \chi_k(e^{itH_0}m)=e^{2ikt} \qquad (t\in\RR,m\in M_c). $$
For $k\in\ZZ$ and $\lambda\in i\fraka^*$ we form the principal series representation $\pi_{k,\lambda}$ (normalized parabolic induction)
$$ \pi_{k,\lambda} = \Ind_{P}^G(\chi_k\otimes e^\lambda\otimes1). $$
Then the Plancherel formula for $\calO_G$ reads (see e.g. \cite[Theorem 10.21]{vdB05})
$$ L^2(\calO_G) \simeq \widehat{\bigoplus_{k\in\ZZ}}\int_{i\fraka^*_+}^\oplus\pi_{k,\lambda}\,d\lambda, $$
where we write $i\fraka^*_+$ for all $\lambda=\mu\alpha$ with $\mu\in i\RR_+$.

Just as in Lemma~\ref{lem:CasimirEigenvalues} one shows:

\begin{lemma}\label{lem:EigenvaluesOfCandDcomplex}
    The Casimir elements $C,D\in U(\frakg)$ act on $\pi_{k,\lambda}$ by the scalars
    $$ -\frac{1}{32}(\mu^2-\rho^2)-\frac{1}{8}k^2, \qquad \mbox{and} \qquad -\frac{1}{8}i\mu k, $$
    where $\mu=\lambda(H_0)$ and $\rho=\rho_\fraka(H_0)=rd-2$.
\end{lemma}

In view of Theorem~\ref{thm:UnitaryDualPGL}, this leads to the following explicit $\theta$ correspondence. 

\begin{corollary}\label{cor:ExplicitThetaComplex}
    For $k\in\ZZ$ and $\lambda\in i\fraka^*$, we have $\theta(\pi_{k,\lambda})=\tau_{k,\nu}$ with $\nu=\frac{1}{2}\lambda(H_0)\in i\RR$.
\end{corollary}

\begin{proof}
    Substituting the eigenvalues of $C$ and $D$ from Lemma~\ref{lem:EigenvaluesOfCandDcomplex} into Theorem~\ref{thm:AbstractTheta} shows the claim.
\end{proof}

\section{Examples}\label{sec:Tables}

We list the Jordan algebras together with the Lie algebras of some of the relevant groups as well as the structure constants $r$ and $d$ that show up in the theta correspondence.

\subsection{$V$ Euclidean}

All entries of the table except $\frakg^\sigma$ can be found in the table in \cite[Chapter V.3]{FK94}. The information about $\frakg^\sigma$ can easily be computed for the classical cases, while in the exceptional case there are only two symmetric subalgebras of $\frakf_4$ and only one of them has the correct dimension.

\begin{center}
\begin{tabular}{||c c c c c c||} 
 \hline
 $V$ & $\mathfrak{co}(V)$ & $\frakg$ & $\frakg^\sigma$ & $r$ & $d$\\
 \hline\hline
 $\Sym(r,\RR)$ & $\sp(r,\RR)$ & $\so(r)$ & $\so(r-1)$ & $r$ & $1$\\ 
 \hline
 $\Herm(r,\CC)$ & $\su(r,r)$ & $\su(r)$ & $\fraks(\fraku(1)+\fraku(r-1))$ & $r$ & $2$\\
 \hline
 $\Herm(r,\HH)$ & $\so^*(4r)$ & $\sp(r)$ & $\sp(1)+\sp(r-1)$ & $r$ & $4$\\
 \hline
 $\RR^{1,n-1}$ & $\so(2,n)$ & $\so(n-1)$ & $\so(n-2)$ & $2$ & $n-2$\\
 \hline
 $\Herm(3,\OO)$ & $\frake_{7(-25)}$ & $\frakf_4$ & $\so(9)$ & $3$ & $8$\\
 \hline
\end{tabular}
\end{center}

\subsection{$V$ non-Euclidean}

In the classical cases this is again an easy computation. In the exceptional case there is only one symmetric space of rank one for $\frakf_{4(4)}$.

\begin{center}
\begin{tabular}{||c c c c c c||} 
 \hline
 $V$ & $\mathfrak{co}(V)$ & $\frakg$ & $\frakg^\sigma$ & $r$ & $d$\\
 \hline\hline
 $M(r,\RR)$ & $\sl(2r,\RR)$ & $\sl(r,\RR)$ & $\fraks(\gl(1,\RR)+\gl(r-1,\RR))$ & $r$ & $2$\\
 \hline
 $\Skew(2r,\RR)$ & $\so(2r,2r)$ & $\sp(r,\RR)$ & $\sp(1,\RR)+\sp(r-1,\RR)$ & $r$ & $4$\\
 \hline
 $\RR^{p,q}$ & $\so(p+1,q+1)$ & $\so(p-1,q)$ & $\so(p-1,q-1)$ & $2$ & $p+q-2$\\
 \hline
 $\Herm(3,\OO_s)$ & $\frake_{7(7)}$ & $\frakf_{4(4)}$ & $\so(4,5)$ & $3$ & $8$\\
 \hline
\end{tabular}
\end{center}

Note that for $V=\RR^{p,q}$ we assume that $p+q$ is even, otherwise no group with Lie algebra $\co(V)=\so(p+1,q+1)$ has a minimal representation (at least for $p,q>2$).

\subsection{$V$ complex}

Both $V$, $\co(V)$, $\frakg$ and $\frakg^\sigma$ are the complexifications of the corresponding objects for Euclidean $V$.

\begin{center}
\begin{tabular}{||c c c c c c||} 
 \hline
 $V$ & $\mathfrak{co}(V)$ & $\frakg$ & $\frakg^\sigma$ & $r$ & $d$\\
 \hline\hline
 $\Sym(r,\CC)$ & $\sp(r,\CC)$ & $\so(r,\CC)$ & $\so(r-1,\CC)$ & $r$ & $1$\\
 \hline
 $M(r,\CC)$ & $\sl(2r,\CC)$ & $\sl(r,\CC)$ & $\fraks(\gl(1,\CC)+\gl(r-1,\CC))$ & $r$ & $2$\\
 \hline
 $\Skew(2r,\CC)$ & $\so(4r,\CC)$ & $\sp(r,\CC)$ & $\sp(1,\CC)+\sp(r-1,\CC)$ & $r$ & $4$\\
 \hline
 $\CC^n$ & $\so(n+2,\CC)$ & $\so(n-1,\CC)$ & $\so(n-2,\CC)$ & $2$ & $n-2$\\
 \hline
 $\Herm(3,\OO)_\CC$ & $\frake_7(\CC)$ & $\frakf_4(\CC)$ & $\so(9,\CC)$ & $3$ & $8$\\
 \hline
\end{tabular}
\end{center}

\bibliographystyle{amsplain}
\bibliography{bibdb}

\bigskip

\textsc{Department of Mathematics, Aarhus University, Ny Munkegade 118, 8000 Aarhus, Denmark}

\textit{E-mail address:} \texttt{frahm@math.au.dk}\\

\textsc{Centre de Recherches Math\'{e}matiques, Universit\'{e} de Montr\'{e}al,
P.O. Box 6128, Centre-ville Station, Montr\'{e}al (Qu\'{e}bec), Canada}

\textit{E-mail address:} \texttt{labriet.quentin@gmail.com}

\end{document}